\documentclass[12pt]{amsart}
\usepackage{amscd}
\usepackage{amssymb}
\usepackage[T1]{fontenc}
\usepackage{eepic}
\usepackage[french]{babel}
\usepackage{times}
\usepackage[latin1]{inputenc}

\allowdisplaybreaks[1]

\newtheorem{thm}{Th\'eor\`eme}[section]
\newtheorem{remark}[thm]{Remarque}

\newtheorem{lem}[thm]{Lemme}
\newtheorem{proposition}[thm]{Proposition}

\newtheorem{definition}[thm]{D\'efinition}
\newtheorem{ex}[thm]{Exemple}
\theoremstyle{definition}

\numberwithin{equation}{section}

\newcommand{\Z}{\mathbb Z}

\newcommand{\n}{\noindent}

\newcommand{\cal}{\mathcal}

\newcommand{\dis}{\displaystyle}

\newcommand{\resumename}{Abstract}
\newenvironment{resume}{\narrower\footnotesize\bf
\noindent\resumename.\quad\footnotesize\rm}{\par\bigskip}

\font\hb=cmbx12
\newcommand\pt{\hbox{\hb .}}

\begin{document}

\title[Les pré-$(a,b)$-algèbres]{ Les pré-$(a,b)$-algèbres à homotopie près}

\author[Walid Aloulou]{Walid Aloulou}
\date{23/06/2012}

\begin{abstract}

Dans cet article on \'etudie le concept d'alg\`ebre \`a homotopie
pr\`es pour une structure d\'efinie par deux op\'erations
$\curlywedge$ et $\lozenge$. Des exemples importants d'une telle
structure sont ceux des alg\`ebres pré-Gerstenhaber et
pré-Poisson. Etant donn\'ee une structure d'alg\`ebre
pré-commutative et pré-Lie gradu\'ee pour deux d\'ecalages des
degr\'es donn\'es par $a$ et $b$, on définit la structure d'une
pré-$(a,b)$-algèbre et on donne une construction explicite de
l'alg\`ebre \`a homotopie pr\`es associ\'ee.

\end{abstract}


\address{Universit\'e de Sousse, Laboratoire de Mathématique Physique Fonctions Spéciales et
Applications. Université de Sfax, D\'epartement de
Math\'ematiques, Institut Pr\'eparatoire aux Etudes d'Ing\'enieurs
de Sfax, Route Menzel Chaker Km 0.5, Sfax, 3018, Tunisie}
\email{Walid.Aloulou@ipeim.rnu.tn}


\thanks{}

\maketitle


\begin{resume}

We study in this article the concept of algebra up to homotopy for
a structure defined by two operations $ \curlywedge $ and $
\lozenge $. Important examples of such structure are those of
pre-Gerstenhaber and pre-Poisson algebras.

Given a structure of pre-commutative and pre-Lie algebra for two
shifts of degree given by $a$ and  $b$, we define the structure of
a pre-$(a, b)$-algebra and we give an explicit construction of the
associated algebra up to homotopy

\end{resume}

\

\section{Introduction}\label{sec1}

Soit $\mathcal A$ une alg\`ebre munie d'une op\'eration $m$ ($m$
est associative, ou commutative, ou Lie...). On dira juste que
$\mathcal A$ est une $\mathcal{P}$-algèbre ou une alg\`ebre
d'opérade $\mathcal P$ ($\mathcal{P}$ est $Ass$ ou $Com$ ou
$Lie$...). Dans beaucoup de cas, on sait d\'efinir la notion
d'alg\`ebre d'opérade $\mathcal P$ \`a homotopie pr\`es de
$\mathcal A$. Pr\'ecis\'ement, si $\mathcal{A}$ est une $\mathcal
P$-algèbre, on lui associe canoniquement une cog\`ebre gradu\'ee
$\big(\mathcal{C}(\mathcal{A}),\Delta\big)$. Une structure de
$\mathcal{P}$-alg\`ebre \`a homotopie pr\`es sur $\mathcal{A}$ est
\'equivalente \`a la donn\'ee d'une cod\'erivation
$Q:\big(\mathcal{C}(\mathcal{A}),\Delta\big)
\longrightarrow\big(\mathcal{C}(\mathcal{A}),\Delta\big)$ de
degr\'e $1$ et de carr\'e nul (c'est \`a dire que $Q$ est une
codiff\'erentielle). La cog\`ebre codiff\'erentielle
$\big(\mathcal{C}(\mathcal{A}),\Delta,Q\big)$ est appel\'ee la
$\mathcal{P}$-alg\`ebre \`a homotopie pr\`es de $\mathcal{A}$.
Cette alg\`ebre donne naturellement les complexes d'homologie et
de cohomologie associ\'es \`a ce type d'alg\`ebre pour $\mathcal
A$ et ses modules (voir \cite{[AAC]}, \cite{[C]}). Par exemple, si
$\mathcal{A}$ est une $Lie$-algèbre, alors on sait construire
l'algèbre de Lie à homotopie près associée et retrouver l'homologie et
 la cohomologie de Chevalley-Eilenberg (des algèbres de Lie).\\

Lorsque $\mathcal A$ poss\`ede deux op\'erations avec des
relations de compatibilit\'e, la construction de l'alg\`ebre \`a
homotopie pr\`es enveloppante correpondante est plus difficile. On
peut citer dans ce cadre les $(a,b)$-algèbres (voir \cite{[A]}).
En particulier une $(0,0)$-algèbre est une algèbre de Poisson
graduée et une $(0,-1)$-algèbre est une algèbre de Gerstenhaber.\\

Un autre exemple d'algèbre à deux opérations est l'algèbre
pré-Gerstenhaber à droite $(\mathcal{G},\curlywedge,\lozenge)$
définie dans \cite{[AAC2]} par :
\begin{itemize}
\item $(\mathcal{G},\curlywedge)$ est une algèbre de Zinbiel \`a droite graduée, $|\curlywedge|=0$.

\item $(\mathcal{G}[1],\lozenge)$ est une algèbre pré-Lie \`a droite gradu\'ee, $|\lozenge|=-1$.
\item  Les relations de compatibilité entre $\curlywedge$ et $\lozenge$ sont :
$$\aligned
&\alpha\curlywedge(\beta\lozenge \gamma)=(-1)^{(|\beta|-1)(|\gamma|-1)}\alpha\curlywedge(\gamma\lozenge \beta)\\
&\alpha\lozenge(\beta\curlywedge \gamma)=(\alpha\lozenge \beta)\curlywedge\gamma\\
&(\alpha\lozenge\beta)\curlywedge
\gamma=(-1)^{(|\beta|-1)|\gamma|}(\alpha\curlywedge
\gamma)\lozenge\beta.
\endaligned
$$
\end{itemize}

Rappelons qu'une alg\`ebre pr\'e-Lie est un espace vectoriel $V$
muni d'une loi $\lozenge$ telle que son antisym\'etris\'ee est une
loi d'alg\`ebre de Lie. Il existe donc une notion d'alg\`ebre
pr\'e-Lie \`a homotopie pr\`es (\cite{[ChL]}). De m\^eme, une
alg\`ebre pr\'e-commutative, appel\'ee aussi alg\`ebre de Zinbiel
est \'equipp\'ee d'un produit $\curlywedge$, dont le sym\'etris\'e
est associatif et commutatif. Il existe donc une notion d'alg\`ebre de Zinbiel \`a homotopie pr\`es (\cite{[Liv]}).\\

Remarquons que si $(\mathcal{G},\curlywedge,\lozenge)$ est une
algèbre pré-Gerstenhaber, alors si on symétrise $\curlywedge$ et
on antisymétrise $\lozenge$, on obtient une algèbre de
Gerstenhaber.\\

Le pr\'esent travail consiste \`a unifier les constructions
d'alg\`ebre \`a homotopie pr\`es dans les cas des alg\`ebres
pré-Poisson et des alg\`ebres pré-Gerstenhaber, ce qui nous permet
de d\'efinir la structure d'une pré-$(a,b)$-alg\`ebre \`a
homotopie pr\`es. Disons qu'une pré-$(a,b)$-alg\`ebre est un
espace vectoriel gradu\'e $\mathcal{A}$ muni de deux produits
 $\curlywedge$ de degr\'e $a\in\mathbb{Z}$ $(|\curlywedge|=a)$ et $\lozenge$ de degr\'e $b\in\mathbb{Z}$
  $(|\lozenge|=b)$ tel que
$\big(\mathcal{A}[-a],\curlywedge\big)$ est une alg\`ebre de
Zinbiel gradu\'ee et $\big(\mathcal{A}[-b],\lozenge\big)$ est une
alg\`ebre pré-Lie gradu\'ee. Ces deux produits v\'erifient des
relations de compatibilit\'e entre eux donn\'ees par:
$$\aligned
&\alpha\curlywedge(\beta\lozenge \gamma)=(-1)^{(|\beta|+b)(|\gamma|+b)}\alpha\curlywedge(\gamma\lozenge \beta)\\
&\alpha\lozenge(\beta\curlywedge \gamma)=(\alpha\lozenge \beta)\curlywedge\gamma\\
&(\alpha\lozenge\beta)\curlywedge
\gamma=(-1)^{(|\beta|+b)(|\gamma|+a)}(\alpha\curlywedge
\gamma)\lozenge\beta.
\endaligned
$$

Si on pose
$[\alpha,\beta]=\alpha\lozenge\beta-(-1)^{(|\alpha|+b)(|\beta|+b)}\beta\lozenge\alpha$
et
$\alpha\pt\beta=\alpha\curlywedge\beta+(-1)^{(|\alpha|+a)(|\beta|+a)}\beta\curlywedge\alpha$,
on obtient que $(\mathcal A,\pt,[~,~])$ est une $(a,b)$-alg\`ebre.

Dans le cas o\`u $a=0$ et $b=-1$, on retrouve les alg\`ebres
pré-Gerstenhaber et dans le cas o\`u $a=b=0$, on trouve une
alg\`ebre qu'il est naturel d'appeler algèbre pré-Poisson gradu\'ees.\\


\section{G\'en\'eralit\'es}

\

Soit $V=\bigoplus_{n\in\Z}V_n$ un espace vectoriel $\mathbb{Z}$ gradué. Le degr\'e d'un
élément homogène $x$ dans $V$ est noté $|x|$. On notera $T^+(V)$ l'espace $\bigoplus_{n>0}\bigotimes^n V$, gradu\'e par $|x_1\otimes\dots\otimes x_n|=|x_1|+\dots+|x_n|$.\\

\begin{definition}

\

\noindent 1) Une algèbre graduée est un espace vectoriel gradué
$V$ muni d'une application bilinéaire $ m: V \otimes V \rightarrow
V$ de degré 0 ($|m|=0$) c'est \`a dire :
$$m(V_i\otimes V_j)\subset V_{i+j}.$$

\noindent 2) Une dérivation $d$ de l'algèbre $(V, m)$ est une
application vérifiant:
$$d \circ m = m\circ(d \otimes id + id \otimes d).$$ Si $d$ est de degré $1$ et
$d^2=0$, on dit que $d$ est une différentielle de $(V,m)$.\\
\end{definition}

\begin{definition}

\

\noindent 1) Une cogèbre graduée est un espace vectoriel gradué
$\cal C$ muni d'une application linéaire $ \Delta: \cal
C\longrightarrow \cal C \otimes \cal C$ dite comultiplication de
$\cal C$ vérifiant:
$$\Delta {\cal C}_k \subset \dis\sum_{i+j =k} {\cal C}_i\otimes {\cal C}_j$$

\noindent 2) Une codérivation $Q$ de la cogèbre $(\mathcal{C},
\Delta)$ est une application vérifiant:
$$
\Delta \circ Q = (Q \otimes id + id \otimes Q)\circ \Delta.
$$
Si $Q$ est de degré $1$ et $Q^2=0$, on dit que $Q$ est une
codifférentielle de $(\mathcal{C},\Delta)$. Dans ce cas, on dit
que $(\mathcal C,\Delta,Q)$ est une cog\`ebre codiff\'erentielle.\\
\end{definition}

Par définition, l'espace $V[1]$ est le m\^eme espace que $V$, mais
avec un d\'ecalage du degr\'e : le degré d'un élément homogène $x$
dans $V[1]$ noté $deg(x)$ devient $deg(x)=|x|-1$.\\

Les bonnes structures algèbriques sont des lois associées à une
opérade quadratique $\mathcal P$ (\cite{[GK]}). En effet, si $V$
est un espace vectoriel gradué, on peut dans ce cas construire la
cogèbre colibre $(W,\Delta)$ sur l'opérade duale $\mathcal P^!$
engendré par le décalé $V[1]$ de $V$.

Dans cette situation, dire qu'une loi $m:V\otimes V\rightarrow V$ de degré 0 est une structure de type $\mathcal P$,
c'est dire que sa décalée $m':V[1]\otimes V[1]\rightarrow V[1]$ est bilin\'eaire, de degré 1 et
 vérifie les symétries associées à l'opérade $\mathcal P^!$, donc est prolongeable de façon unique en
  une codérivation $Q$ de $(W,\Delta)$ et la loi $m$ vérifie les axiomes de la structure si et seulement si,
  $Q$ vérifie l'équation de structure $[Q,Q]=2Q^2=0$.\\

Toujours dans ce cas, on parlera de $\mathcal P_\infty$ alg\`ebre ou d'alg\`ebre \`a homotopie pr\`es pour
toute cog\`ebre codiff\'erentielle $(W,\Delta,Q)$ correspondante \`a $(W,\Delta)$.\\

\begin{definition}{\rm \cite{[GK]}}

\

Une structure de $\mathcal{P}_\infty$ algèbre sur un espace vectoriel $V$ est définie
par la donnée d'une codifférentielle $Q$ sur la $\mathcal{P}^!$-cogèbre $(W,\Delta)$ construite \`a partir de $V$.\\
\end{definition}

En particulier, si $(V,m,d)$ est une alg\`ebre diff\'erentielle, on peut prolonger $d+m'$ en une unique cod\'erivation
$Q$ de degr\'e 1 de $(W,\Delta)$ telle que $Q^2=0$, c'est \`a dire $(W,\Delta,Q)$ est une $\mathcal P_\infty$ alg\`ebre.
 D\'ecrivons quelques exemples.\\



\section{Les $(a,b)$-alg\`ebres \`a homotopie pr\`es}\label{sec4}

Cette section consiste \`a unifier les constructions d'alg\`ebre
\`a homotopie pr\`es dans les cas des alg\`ebres de Poisson et des
alg\`ebres de Gerstenhaber, ce qui nous permet de d\'efinir la
structure d'une $(a,b)$-alg\`ebre \`a homotopie pr\`es.

\subsection{Définitions et notations}

\begin{definition}

\

Consid\'erons un espace vectoriel $\mathcal{A}$ gradu\'e. Le
degr\'e d'un \'el\'ement homog\`ene $\alpha$ de $\mathcal{A}$ est
not\'e $|\alpha|$. Soient $a,b\in \mathbb{Z}$, l'espace
$\mathcal{A}$ est muni d'un produit $\pt$ de degr\'e $a$
($|\pt|=a$) et d'un crochet $[~,~]$ de degr\'e $b$ ($|[~,~]|=b$)
tel que $\Big(\mathcal{A}[-a],\pt\Big)$ est une alg\`ebre
commutative et associative gradu\'ee  et
$\Big(\mathcal{A}[-b],[~,~]\Big)$ est une alg\`ebre de Lie
gradu\'ee. De plus, l'application lin\'eaire
$ad:\mathcal{A}[-b]\longrightarrow\mathcal{D}er\Big(\mathcal{A}[-a],
\pt\Big)$; $\alpha\longmapsto ad_{\alpha}$
 est telle que $ad_{\alpha}$ soit une d\'erivation gradu\'ee pour le produit
 $\pt$.

 On dit que $\Big(\mathcal{A},\pt,[~,~]\Big)$ est une
 $(a,b)$-alg\`ebre gradu\'ee. Pour tout $\alpha,\beta,\gamma\in\mathcal{A}$, on a les propri\'et\'es
 suivantes:\\

 {\bf(i)} $
 \alpha\pt\beta=(-1)^{(|\alpha|+a)(|\beta|+a)}\beta\pt\alpha$,\\

{\bf(ii)}
$\alpha\pt(\beta\pt\gamma)=(\alpha\pt\beta)\pt\gamma$,\\

{\bf(iii)}
$[\alpha,\beta]=-(-1)^{(|\alpha|+b)(|\beta|+b)}[\beta,\alpha]$,\\

{\bf(iv)}
$(-1)^{(|\alpha|+b)(|\gamma|+b)}[[\alpha,\beta],\gamma]+(-1)^{(|\beta|+b)(|\alpha|+b)}\big[[\beta,\gamma]
,\alpha\big]$

\hskip5.07cm$+(-1)^{(|\gamma|+b)(|\beta|+b)}\big[[\gamma,\alpha],\beta\big]=0$,\\

{\bf (v)} $[\alpha,\beta\pt\gamma]=[\alpha,\beta]\pt
\gamma+(-1)^{(|\beta|+a)(|\alpha|+b)}\beta\pt[\alpha,\gamma] $\\

qui s'\'ecrit encore $[\alpha\pt\beta,\gamma]=\alpha\pt[\beta,
\gamma]+(-1)^{(|\beta|+a)(|\gamma|+b)}[\alpha,\gamma]\pt\beta $.\\

\noindent De plus, si on a une diff\'erentielle
$d:\mathcal{A}[-a]\longrightarrow\mathcal{A}[-a+1]$ \Big(ou
$d:\mathcal{A}[-b]\longrightarrow\mathcal{A}[-b+1]$\Big) de
degr\'e $1$ v\'erifiant $d\circ d=0,$
$$d(\alpha\pt
\beta)=d\alpha\pt\beta+(-1)^{|\alpha|+a}\alpha\pt d\beta \
\hbox{et} \
d([\alpha,\beta])=[d\alpha,\beta]+(-1)^{|\alpha|+b}[\alpha,d\beta]
,$$ on dira que $\Big(\mathcal{A},\pt,[~,~],d\Big)$ est une
 $(a,b)$-alg\`ebre diff\'erentielle gradu\'ee.\end{definition}

On utilise un d\'ecalage pour homog\`en\'eiser le produit et la
diff\'erentielle.
 On consid\`ere l'espace $\mathcal{A}[-a+1]$ muni de la graduation
 $dg(\alpha)=|\alpha|+a-1$ que l'on note simplement par $\alpha$. Sur $\mathcal{A}[-a+1]$, le produit
 $\pt$ n'est plus commutatif et le crochet $[~,~]$ n'est plus
 antisym\'etrique. On construit, donc, un nouveau produit $\mu$ sur $\mathcal{A}[-a+1]=\mathcal{A}[-a][1]$ de degr\'e $1$ d\'efini
 par
 $$\mu(\alpha,\beta)=(-1)^{1.\alpha}\alpha\pt\beta$$
et un nouveau crochet $\ell$ sur
$\mathcal{A}[-a+1]=\mathcal{A}[-b][b-a+1]$ de degr\'e $b-a+1$
d\'efini par
$$\ell(\alpha,\beta)=(-1)^{(b-a+1).\alpha}[\alpha,\beta].$$ Et on a

\noindent~{\bf(i)} $\mu(\alpha,\beta)=-(-1)^{\alpha
\beta}\mu(\beta,\alpha)$,\\

\noindent~{\bf(ii)}
$\mu\left(\mu(\alpha,\beta),\gamma\right)=-(-1)^{\alpha}\mu\left(\alpha,\mu(\beta,\gamma)\right)$,\\

\noindent~{\bf(iii)} $\ell(\alpha,\beta)=-(-1)^{b-a+1}(-1)^{\alpha
\beta}\ell(\beta,\alpha)$,\\

\noindent~{\bf(iv)} $(-1)^{\alpha
\gamma}\ell(\ell(\alpha,\beta),\gamma\big)+(-1)^{\beta
\alpha}\ell(\ell(\beta,\gamma),\alpha) +(-1)^{\gamma
\beta}\ell(\ell(\gamma,\alpha),\beta)=0$,\\

 \noindent~{\bf(v)} \
$\ell(\alpha,\mu(\beta,\gamma))=
(-1)^{\alpha+b-a+1}\mu(\ell(\alpha,\beta),\gamma)+
(-1)^{(\alpha+b-a+1)(\beta+1)} \mu(\beta,\ell(\alpha,\gamma)),$\\

ou encore
$$\ell(\alpha,\mu(\beta,\gamma))=(-1)^{(b-a+1)(\alpha+1)}\mu(\alpha,\ell(\beta,
\gamma))+(-1)^{b-a+1+\beta \gamma}\mu(\ell(\alpha,\gamma),\beta).
$$

\noindent De plus, $d$ reste encore une d\'erivation pour $\mu$ et
$\ell$, elle v\'erifie:
\begin{align*}&d(\mu(\alpha,
\beta))=-\mu(d\alpha,\beta)+(-1)^{\alpha+1}\mu(\alpha, d\beta) \\&
\hskip-0.8cm\hbox{et} \
d(\ell(\alpha,\beta))=(-1)^{b-a+1}\ell(d\alpha,\beta)+(-1)^{\alpha+b-a+1}\ell(\alpha,d\beta)
.\end{align*}

\subsection{Extension de la multiplication et du crochet \`a la cog\`ebre de Lie codiff\'erentielle}

\

On considère l'espace $\mathcal{A}[-a+1]$ muni du degr\'e
$deg(\alpha)=|\alpha|-1=\alpha$.
 Une permutation $\sigma\in S_{p+q}$ ($p,q\geq1$) est dite un $(p,q)$-shuffle si elle vérifie :
$$
\sigma(1)<\cdots<\sigma(p) \hskip0.4cm \hbox{et} \hskip0.4cm
\sigma(p+1)<\cdots<\sigma(p+q).
$$
On note $Sh(p,q)$ l'ensemble des $(p,q)$-shuffles.

On note aussi
$$
\varepsilon_\alpha\left(\begin{smallmatrix}\alpha_1&\dots&\alpha_n\\
\alpha_{i_1}&\dots&\alpha_{i_n}\end{smallmatrix}\right)=\varepsilon_\alpha(\sigma)
$$
la signature de la permutation $\sigma=\left(\begin{smallmatrix}1&\dots&n\\ i_1&\dots&i_n\end{smallmatrix}\right)$,
en tenant compte des degr\'es de $\alpha_j$, autrement dit, $\varepsilon_\alpha$ est l'unique morphisme de
$S_n$ dans $\mathbb R$ tel que $\varepsilon_\alpha((i,j))=(-1)^{\alpha_i\alpha_j}$.\\

 On d\'efinit ensuite le produit shuffle sur $\bigotimes^+ \mathcal{A}[-a+1]$ par :

$$
sh_{p,q}\left(\alpha_1\otimes...\otimes \alpha_p ,
\alpha_{p+1}\otimes...\otimes \alpha_{p+q}\right)=\sum_{\sigma\in
Sh(p,q)} \varepsilon_{\alpha}(\sigma^{-1})
\alpha_{\sigma^{-1}(1)}\otimes...\otimes
\alpha_{\sigma^{-1}(p+q)}.
$$

 On d\'efinit alors l'espace quotient
$$
\mathcal
H=\underline{\bigotimes}^+\mathcal{A}[-a+1]=\bigoplus_{n\geq1}~^{\bigotimes^n\mathcal{A}[-a+1]}
\diagup_{\sum_{p+q=n}Im(sh_{p,q})}.
$$

Pour
$X=\alpha_1\underline{\otimes}\dots\underline{\otimes}\alpha_n\in\mathcal{H}$,
le degr\'e $dg(X)=\alpha_1+\dots+\alpha_n$ not\'e simplement par
$x$. Sur cet espace, on d\'efinit un cocrochet $\delta$ de degr\'e
$0$  par:
$$
\begin{aligned}
\delta(X)&=\sum_{j=1}^{n-1}\alpha_1\underline{\otimes}\dots\underline{\otimes}
\alpha_j\bigotimes \alpha_{j+1}
\underline{\otimes}\dots\underline{\otimes} \alpha_n\\&
\hskip1cm-\varepsilon_\alpha\left(\begin{smallmatrix}\alpha_1\dots
\alpha_{j}&\alpha_{j+1}\dots \alpha_n\\ \alpha_{j+1}\dots
\alpha_n&\alpha_1\dots
\alpha_j\end{smallmatrix}\right)\alpha_{j+1}\underline{\otimes}\dots\underline{\otimes}
\alpha_n\bigotimes
\alpha_1\underline{\otimes}\dots\underline{\otimes}
\alpha_j\\&=\sum_{U\underline\otimes V=X\atop U,V\neq\emptyset}
U\bigotimes V-(-1)^{v u}V\bigotimes U.
\end{aligned}
$$

On prolonge $\mu$ et $d$ \`a $\mathcal{H}$ comme des
cod\'erivations $\mu_1$ et $d_1$ de $\delta$ de degr\'e $1$ en
posant:
$$
d_1(\alpha_1\underline\otimes \dots \underline\otimes
\alpha_{n})=\sum_{1\leq k\leq
n}(-1)^{\sum_{i<k}\alpha_i}\alpha_1\underline\otimes \dots
\underline\otimes d(\alpha_k)\underline\otimes
\dots\underline\otimes \alpha_{n}
$$et
$$
\mu_1(\alpha_1\underline\otimes \dots \underline\otimes
\alpha_{n})=\sum_{1\leq
k<n}(-1)^{\sum_{i<k}\alpha_i}\alpha_1\underline\otimes \dots
\underline\otimes \mu(\alpha_k,\alpha_{k+1})\underline\otimes
\dots\underline\otimes \alpha_{n}.
$$
Alors, $$(\mu_1\otimes
id+id\otimes\mu_1)\circ\delta=\delta\circ\mu_1, \mu_1^2=0,
(d_1\otimes id+id\otimes d_1)\circ\delta=\delta\circ d_1 \
\hbox{et} \ d_1^2=0.$$ (Voir \cite{[AAC1]})
\\

En posant $D_1=d$, $D_2=\mu$, $D_k=0$, si $k\geq3$ et
\begin{align*}&D(\alpha_1\underline{\otimes}\dots\underline{\otimes}
\alpha_n)=\\&\sum_{1\leq r\leq n\atop 0\leq j\leq
n-r}(-1)^{\sum_{i\leq
j}\alpha_i}\alpha_1\underline{\otimes}\dots\underline{\otimes}
\alpha_j\underline\otimes
D_r(\alpha_{j+1}\underline{\otimes}\dots\underline{\otimes}
\alpha_{j+r})\underline{\otimes}
\alpha_{j+r+1}\underline\otimes\dots\underline{\otimes}
\alpha_n.\end{align*}

Alors, $D=d_1+\mu_1$ est l'unique cod\'erivation de $\delta$ de
degr\'e $1$ qui prolonge $d$ et $\mu$ \`a $\mathcal{H}$.

Elle v\'erifie $$D\circ D=0 \ \hbox{et} \ (D\otimes id+id\otimes
D)\circ\delta=\delta\circ D.$$ On obtient que
$\big(\mathcal{H},\delta,D\big)$ est une cog\`ebre de Lie
codiff\'erentielle, donc, c'est une $C_\infty$
alg\`ebre.\\

On prolonge, ensuite, le crochet $\ell$ \`a $\mathcal{H}$.

\begin{proposition}

\

Sur $\mathcal{H}$, il existe un unique "crochet" $\ell_2$, de
degr\'e $b-a+1$, v\'erifiant:
\begin{equation}
\delta\circ\ell_2=(\ell_2\otimes
id)\circ\big(\tau_{23}\circ(\delta\otimes id)
+id\otimes\delta\big)+(id\otimes\ell_2)\circ\big(\delta\otimes
id+\tau_{12}\circ(id\otimes \delta)\big).
\end{equation}
Ce crochet est d\'efini pour
$X=\alpha_1\underline{\otimes}\dots\underline{\otimes} \alpha_p$
et $Y=\alpha_{p+1}\underline{\otimes}\dots\underline{\otimes}
\alpha_{p+q}$ par:
\begin{align*} \ell_2(X,Y)=&\hskip-0.6cm\sum_{\sigma\in Sh(p,q)\atop
k,\sigma^{-1}(k)\leq p<\sigma^{-1}(k+1)}\hskip-0.6cm
\varepsilon_\alpha(\sigma^{-1})(-1)^{(b-a+1)\sum_{s<k}\alpha_{\sigma^{-1}(s)}}\times\\&\hskip1.4cm\times
\alpha_{\sigma^{-1}(1)}\underline{\otimes}\dots\underline{\otimes}\ell(\alpha_{\sigma^{-1}(k)},
\alpha_{\sigma^{-1}(k+1)})\underline{\otimes}
\dots\underline{\otimes}\alpha_{\sigma^{-1}(p+q)}.\end{align*}

\end{proposition}

\subsection{Alg\`ebre de Lie diff\'erentielle gradu\'ee associ\'ee \`a une $(a,b)$-alg\`ebre diff\'erentielle}

\

On consid\`ere, maintenant, l'espace $\mathcal{H}[a-b-1]$ muni de
la graduation $dg'(X)=dg(X)-a+b+1$ not\'e simplement par $x'$ pour
$X\in\mathcal{H}[a-b-1]$. On pose
$\ell'_2(X,Y)=(-1)^{(a-b-1)dg'(X)}\ell_2(X,Y)$. Alors, le crochet
$\ell'_2$ est de degr\'e $0$ dans $\mathcal{H}[a-b-1]$ et la
diff\'erentielle $D$ reste de degr\'e $1$. Et on a

\begin{proposition}

\

L'espace $\mathcal H[a-b-1]$, muni du crochet $\ell'_2$ et de la
diff\'erentielle
 $D$ est une alg\`ebre de Lie diff\'erentielle
gradu\'ee: Pour tout $X$, $Y$ et $Z$ de $\mathcal H[a-b-1]$, on a:
\begin{align*}
&{\bf(i)} \quad\ell'_2(X,Y)=-(-1)^{x'y'}\ell'_2(Y,X),\\& {\bf(ii)}
\quad
(-1)^{x'z'}\ell'_2\Big(\ell'_2(X,Y),Z\Big)+(-1)^{y'x'}\ell'_2\Big(\ell'_2(Y,Z),X\Big)
\\& \hskip4.9cm+(-1)^{z'y'}\ell'_2\Big(\ell'_2(Z,X),Y\Big)=0,\\&
{\bf(iii)} \quad
D\Big(\ell'_2(X,Y)\Big)=\ell'_2\Big(D(X),Y\Big)+(-1)^{x'}\ell'_2\Big(X,D(Y)\Big).
\end{align*}
\end{proposition}

\subsection{La $L_\infty$ alg\`ebre $S^+({\mathcal H}[a-b])$}

\

Dans le paragraphe pr\'ec\'edent, on a montr\'e que $\Big(\mathcal
H[a-b-1], \ell'_2,D\Big)$ est une alg\`ebre de Lie
diff\'erentielle gradu\'ee. On consid\`ere l'espace
$\mathcal{H}[a-b]$ muni de la graduation
$$dg''(X)=dg'(X)-1=dg(X)-a+b:=x''\ , \ \hbox{pour tout} \ X\in\mathcal{H}[a-b].$$

On voudrait construire la cog\`ebre cocommutative coassociative
$(S^+(\mathcal{H}[a-b]),\Delta)$, o\`u
$S^+(\mathcal{H}[a-b])=\bigoplus_{n\geq1}S^n(\mathcal{H}[a-b])$ et
$\Delta$ est son coproduit qui est de degr\'e $0$ et d\'efini par:

$\forall X_1\dots X_n\in S^n(\mathcal{H}[a-b]) $,
\begin{align*}\Delta(X_1\dots X_n)=\sum_{I\cup J=\{1,\dots n\}\atop \#I, \#J>0}
\varepsilon_{x''}\left(\begin{smallmatrix}x_1\dots x_n\\
x_I x_J\end{smallmatrix}\right)X_I\bigotimes X_J.\end{align*} Le
crochet $\ell'_2$ \'etait antisym\'etrique de degr\'e $0$ sur
$\mathcal{H}[a-b-1]$. Comme l'on veut une cod\'erivation de
degr\'e $1$ pour $\Delta$, on pose
$\ell''_2(X,Y)=(-1)^{x''}\ell'_2(X,Y)$ qui est une application
sym\'etrique sur $\mathcal{H}[a-b]$ de degr\'e $1$. On a

\begin{proposition}

\

Pour tout $X$, $Y$, $Z\in\mathcal H[a-b]$, on a:
\begin{align*}&{\bf (i)} \
 \ell''_2(X,Y)=(-1)^{x''y''}\ell''_2(Y,X), \cr&{\bf (ii)} \ (-1)^{x''z''}\ell''_2(\ell_2''(X,Y),Z)+(-1)^{y''x''}\ell''_2
 (\ell''_2(Y,Z),X)\\&\hskip4.7cm
 +(-1)^{z''y''}\ell''_2(\ell''_2(Z,X),Y)=0,\cr&{\bf (iii)} \  D(\ell''_2(X,Y))=-\ell''_2(D(X),Y)
+(-1)^{1+x''}\ell''_2(X,D(Y)).\end{align*}

\end{proposition}

\

On prolonge $\ell''_2$ \`a $S^+(\mathcal{H}[a-b])$ de fa\c{c}on
unique comme une cod\'erivation $\ell''$ de $\Delta$ de degr\'e
$1$ en posant:
$$\ell''(X_1\dots
X_n)=\displaystyle\sum_{i<j}\varepsilon_{x''}\left(\begin{smallmatrix}x_1\dots x_n\\
x_i x_jx_1\dots \widehat{ij}\dots x_n
\end{smallmatrix}\right)\ell''_{2}
(X_i,X_j).X_1\dots\widehat{ij}\dots X_n.$$ En utilisant
l'identit\'e de
Jacobi, on peut v\'erifier que $\ell''\circ\ell''=0$.\\

On prolonge, aussi, la diff\'erentielle $D$ \`a
$S^+(\mathcal{H}[a-b])$ comme l'unique cod\'erivation $m$ de
$\Delta$ toujours de degr\'e $1$ en posant:
$$m(X_1\dots
X_n)=\displaystyle\sum_{i=1}^{n}\varepsilon_{x''}\left(\begin{smallmatrix}x_1\dots x_n\\
x_i x_1\dots \widehat{i}\dots x_n
\end{smallmatrix}\right)D(X_i).X_1\dots\widehat{i}\dots
X_n.$$ Elle v\'erifie $m\circ m=0$.

On pose $Q=m+\ell''$, ou $Q_1=D$, $Q_2=\ell''_2$, $Q_k=0$, ~si $
k\geq3$ et
$$Q(X_1\dots X_n)=\displaystyle\sum_{I\cup
J=\{1,\ldots,n\}\atop I\neq\emptyset}\varepsilon_{x''}\left(\begin{smallmatrix}x_1\dots x_n\\
x_I x_J\end{smallmatrix}\right)Q_{\#I} (X_{I}).X_J.$$ Alors, $Q$
v\'erifie $Q^2=0$ et $\left(Q\otimes
id+id\otimes Q\right)\circ\Delta=\Delta\circ Q$.\\

Donc, le complexe $\left(S^+(\mathcal{H}[a-b]),\Delta,Q\right)$
est une cog\`ebre cocommuative coassociative et
codiff\'erentielle, c'est à dire une $L_\infty$
alg\`ebre.\\

\subsection{La $C_\infty$ alg\`ebre $S^+({\mathcal H}[a-b])$}

\

L'espace $\big(\mathcal{H},\delta,D\big)$ \'etant une cog\`ebre de
Lie codiff\'erentielle, on d\'efinit un cocrochet $\delta''$ de
degr\'e $a-b$ sur $\mathcal{H}[a-b]$ par:
$$\delta''(X)=\sum_{U\underline\otimes V=X\atop U,V\neq\emptyset}(-1)^{(a-b)u''}\left( U\bigotimes
V+(-1)^{u''v''+a-b+1}V\bigotimes U\right).$$

On prolonge $\delta''$ \`a $S^+(\mathcal{H}[a-b])$ par:
\begin{align*}\delta''(X_1\dots X_n)=&\hskip-0.4cm\displaystyle\sum_{1\leq s\leq n\atop I\cup
J=\{1,\dots,n\}\setminus\{s\}}\hskip-0.6cm\varepsilon_{x''}\left(\begin{smallmatrix}x_1\dots x_n\\
x_Ix_s x_J\end{smallmatrix}\right)\sum_{U_s\underline\otimes
V_s=X_s\atop U_s,V_s\neq\emptyset
}(-1)^{(a-b)(x_I''+u_s'')}\times\\&\times\left(X_I. U_s\bigotimes
V_s . X_J+(-1)^{u_s''v_s''+a-b+1}X_I. V_s\bigotimes U_s.
X_J\right),\end{align*} qui s'écrit encore
\begin{align*}
&\delta''(X_1\dots X_n)=\sum_{\begin{smallmatrix}1\leq s\leq n\\
I\cup
J=\{1,\dots,n\}\setminus\{s\}\end{smallmatrix}}(-1)^{\sum_{i<s}(a-b)x_i''}\sum_{\begin{smallmatrix}U_s\otimes
V_s=X_s\\ U_s,V_s\neq\emptyset\end{smallmatrix}
}(-1)^{(a-b)u_s''}\times\cr&\times\left(\varepsilon_{x''}\left(\begin{smallmatrix}
x_1\dots x_n\\ x_I~u_s~v_s~x_J\end{smallmatrix}\right)X_I.
U_s\bigotimes \mu V_s .
X_J+(-1)^{a-b+1}\varepsilon_{x''}\left(\begin{smallmatrix} x_1\dots x_n\\
x_I~v_s~u_s~x_J\end{smallmatrix}\right)X_I.\mu V_s\bigotimes U_s .
X_J\right),
\end{align*}
avec \begin{align*}\varepsilon_{x''}\left(\begin{smallmatrix}
x_1\dots x_n\\ x_I~u_s~v_s~x_J\end{smallmatrix}\right)=
\varepsilon_{x''}\left(\begin{smallmatrix} x_1\dots x_n\\
x_I~x_s~x_J\end{smallmatrix}\right)(-1)^{\sum_{i<s\atop i\in
J}x_i''}(-1)^{\sum_{i>s\atop i\in I}x_i''}.
\end{align*}

Alors, $\delta''$ est un cocrochet sur $S^+(\mathcal{H}[a-b])$ de
degr\'e $a-b$. En notant $\tau''$ la volte dans
$S^+(\mathcal{H}[a-b])$, $\delta''$ v\'erifie:

\begin{proposition}
\begin{align*}& i) \
 \tau''\circ\delta''=-(-1)^{a-b}\delta'':  \hbox{$\delta''$ est $(a-b)$-coantisym\'etrique}, \cr& ii) \ \Big(id^{\otimes3}+\tau_{12}''\circ\tau_{23}''+\tau_{23}''\circ\tau_{12}''\Big)\circ(\delta''\otimes
id)\circ\delta''=0:  \hbox{identit\'e de coJacobi},\cr& iii) \
(id\otimes\Delta)\circ\delta''=(\delta''\otimes
id)\circ\Delta+\tau_{12}''\circ(id\otimes\delta'')\circ\Delta:
\hbox{identit\'e de coLeibniz}.\end{align*}

\end{proposition}

Ainsi $(S^+(\mathcal{H}[a-b]),\delta'')$ est une cog\`ebre de Lie.
On montre qu'avec $Q=m+\ell''$, elle est codiff\'erentielle.

\begin{proposition}

\

 $m$ et $\ell''$ sont des cod\'erivations de $\delta''$ de
degr\'e $1$, ils v\'erifient:

{\bf(i)} $\left(m\otimes id+id\otimes
m\right)\circ\delta''=(-1)^{a-b}\delta''\circ m$.

{\bf(ii)} $\left(\ell''\otimes id+id\otimes
\ell''\right)\circ\delta''=(-1)^{a-b}\delta''\circ \ell''$.
\end{proposition}

\noindent Alors, le complexe
$\left(S^+(\mathcal{H}[a-b]),\delta'',Q\right)$ est une cog\`ebre
de Lie codiff\'erentielle gradu\'ee, donc, c'est aussi
une $C_\infty$ alg\`ebre.\\

Enfin, le cocrochet $\delta''$ et le coproduit $\Delta$
v\'erifient l'identit\'e de
coLeibniz:$$(id\otimes\Delta)\circ\delta''=(\delta''\otimes
id)\circ\Delta+\tau_{12}''\circ(id\otimes\delta'')\circ\Delta.
$$ Alors,
$\left(S^+\Big(\mathcal{H}[a-b]\Big),\Delta,\delta'',Q\right)$ est
une bicog\`ebre codiff\'erentielle gradu\'ee.

\begin{definition}

\

  Une $(a,b)$-alg\`ebre \`a homotopie pr\`es sur un espace vectoriel gradu\'e $V$ est
d\'efinie par la donn\'ee d'une codiff\'erentielle $Q$, de degr\'e
$1$ et de carr\'e nul sur la bicog\`ebre $$\Big(S^+
\left(\Big(\underline{\displaystyle\bigotimes}^+V[-a+1]\Big)[a-b]\right),\Delta,\delta''\Big).$$\\
En particulier, si $\mathcal{A}$ est une $(a,b)$-alg\`ebre
diff\'erentielle. Alors, la bicog\`ebre colibre et
codiff\'erentielle
$$\left(\mathcal{C}(\mathcal{A})=S^+
\left(\Big(\underline{\displaystyle\bigotimes}^+\mathcal{A}[-a+1]\Big)[a-b]\right),\Delta,\delta'',Q=\ell''+m\right)$$
est la $(a,b)$-alg\`ebre \`a homotopie pr\`es enveloppante de
$\mathcal{A}$.

\end{definition}

\begin{remark}

\

- Dans le cas o\`u $a=0$, $b=-1$ et $\mathcal{A}$ est une
alg\`ebre de Gerstenhaber diff\'erentielle, on retrouve
l'alg\`ebre de Gerstenhaber \`a homotopie pr\`es enveloppante de
$\mathcal{A}$: $$\left(S^+
\Big(\big(\underline{\displaystyle\bigotimes}^+\mathcal{A}[1]\big)[1]\Big),\Delta,\delta'',Q=\ell''+m\right).$$

- Dans le cas o\`u $a=b=0$ et $\mathcal{A}$ est une alg\`ebre de
Poisson diff\'erentielle grdu\'ee, on retrouve le complexe de
l'alg\`ebre de Poisson \`a homotopie pr\`es enveloppante de
$\mathcal{A}$: $$\left(S^+
\Big(\underline{\displaystyle\bigotimes}^+\mathcal{A}[1]\Big),\Delta,\delta'',Q=\ell''+m\right).$$

Cette construction g\'en\'eralise celle des alg\`ebres de
Gerstenhaber et de Poisson \`a homotopie pr\`es.

\end{remark}


\section{Les alg\`ebres pré-Lie et pré-commutatives à homotopie près}

\subsection{Les alg\`ebres pré-Lie à homotopie
près}

\

La notion d'alg\`ebre pr\'e-Lie a \'et\'e \'etudi\'ee par Livernet et Chapoton (\cite{[Liv],[ChL]}).
Une loi pr\'e-Lie est une loi binaire dont l'antisym\'etris\'e est un crochet de Lie. Plus pr\'ecis\'ement :\\

\begin{definition}

\

Une alg\`ebre pré-Lie (\`a droite) gradu\'ee $(V,\diamond)$ est un
espace gradu\'e $V$ muni d'un produit $\diamond$ de degré $0$
vérifiant :
$$
\forall x,y,z\in V,~~(x\diamond y)\diamond z-x\diamond (y\diamond
z)=(-1)^{|y||z|}\big((x\diamond z)\diamond y -x\diamond(z\diamond
y)\big).
$$

Si de plus, $d:V\longrightarrow V$ est une diff\'erentielle de
degré $1$ telle que
$$
d(x\diamond y)=dx\diamond y+(-1)^{1.|x|}x\diamond dy,
$$
on dira que $(V,\diamond,d)$ est une alg\`ebre pré-Lie diff\'erentielle gradu\'ee.\\
\end{definition}

Cette structure est associ\'ee \`a une op\'erade quadratique,
l'op\'erade $preLie$. L'op\'erade duale, d\'etermin\'ee par
\cite{[ChL]}, est l'op\'erade permutative : $preLie^!=Perm$.

Rappelons qu'une alg\`ebre permutative (\`a droite) $(V,.)$ est un
espace gradu\'e $V$ muni d'un produit . de degré $0$, vérifiant :
$$
\forall x,y,z\in V,~~x. (y. z)=(-1)^{|y||z|}x.(z.y)=(x.y).z.
$$

\begin{definition}

\

Une $preLie^!$-cog\`ebre (ou cogèbre permutative) est un espace
vectoriel gradué $\mathcal{C}$ muni d'une comultiplication
$\Delta:\mathcal{C}\longrightarrow\mathcal{C}\otimes\mathcal{C}$
de degré $0$ vérifiant :
$$
(id\otimes \Delta)\circ \Delta= \tau_{23}\circ(id\otimes
\Delta)\circ \Delta=(\Delta\otimes id)\circ \Delta.
$$
\end{definition}

\begin{proposition} {\rm \cite{[ChL]}}

\

Si $V$ un espace vectoriel gradué. Alors la cog\`ebre permutative
colibre associ\'ee \`a $V[1]$ est $\Big(V[1]\otimes
S(V[1]),\Delta\Big)$ o\`u $\Delta$ est d\'efini par $\Delta
(x\otimes 1)=0$ et :
$$
\Delta(x_0\otimes x_1\dots x_n)=
\displaystyle\sum_{\begin{smallmatrix} 0\leq k\leq n-1 \\
\sigma\in
Sh_{k,1,n-k-1}\end{smallmatrix}}\varepsilon_x(\sigma)x_0\otimes(x_{\sigma(1)}\dots
x_{\sigma(k)})\bigotimes
x_{\sigma(k+1)}\otimes(x_{\sigma(k+2)}\dots x_{\sigma(n)}).
$$
(Ici, $Sh_{k,1,n-k-1}$ est l'ensemble des permutations $\sigma$ de $S_n$ telles que $\sigma(1)<\dots<\sigma(k)$ et $\sigma(k+2)<\dots<\sigma(n)$).\\
\end{proposition}

\begin{remark}

\

Identifions $S^{n+1}(V[1])$ avec un sous espace de $V[1]\otimes
S^n(V[1])$ en posant
$$
x_0\dots x_n=\sum_{\sigma\in S_{n+1}}
\varepsilon_x(\sigma^{-1})x_{\sigma(0)}\otimes x_{\sigma(1)}\dots
x_{\sigma(n)}
$$
On a alors :
\begin{align*}
\Delta (x_0\otimes x_1\dots x_n)&=\sum_{\begin{smallmatrix} I\cup
J=\{1,\dots,n\}\\ J\neq\emptyset
\end{smallmatrix}}\varepsilon_{x}(\begin{smallmatrix} x_1&\dots&x_n\\ x_I&&x_J\end{smallmatrix})(x_0\otimes x_I)\bigotimes x_J\cr
&=x_0\bigotimes x_1\dots x_n+x_0\otimes\Delta'(x_1\dots x_n)
\end{align*}
où $\Delta'(x_1\dots x_n)$ est le coproduit de la cogèbre
cocommutative colibre $S^{+}(V[1])$.

Il est alors clair que $\Delta(x_0\dots x_n)=\Delta'(x_0\dots x_n)$.\\
\end{remark}

\begin{definition}

\

Une $preL_\infty$ alg\`ebre est une cog\`ebre permutative codifférentielle $\Big(V[1]\otimes S(V[1]),\Delta, Q\Big)$ telle que $Q$ est une cod\'erivation de $\Delta$ de degr\'e $1$ et $Q^2=0$.\\
\end{definition}

\begin{proposition} {\rm \cite{[ChL]}}

\

Soit \ $(V, \diamond,d)$ une algèbre pré-Lie diff\'erentielle
gradu\'ee. On pose \vskip0.15cm
$$
Q_1(x)=dx,~~ Q_2(x_1\otimes x_2)=(-1)^{x_1}x_1\diamond
x_2,~~Q_k=0, \forall k\geq3
$$
et
$$\aligned
Q(x_0\otimes x_1\dots x_n)&=Q_1(x_0)\otimes x_1\dots x_n+(-1)^{x_0}\sum_{k=1}^{n}(-1)^{ \sum_{i<k}x_i}x_0\otimes x_1\dots Q_1(x_k)\dots x_n+\\
&\hskip 1cm+\sum_{\sigma\in
Sh_{1,n-1}}\varepsilon_x(\sigma)Q_2(x_0\otimes x_{\sigma(1)})\otimes x_{\sigma(2)}\dots x_{\sigma(n)}+\\
&\hskip 1cm+(-1)^{x_0}\hskip-0.5cm\sum_{\sigma\in
Sh_{1,1,n-2}}\varepsilon_x(\sigma) x_0\otimes
Q_2(x_{\sigma(1)}\otimes x_{\sigma(2)}).x_{\sigma(3)}\dots
x_{\sigma(n)}.
\endaligned
$$

Alors, $\Big(V[1]\otimes S(V[1]),\Delta,Q\Big)$ est une $preL_\infty$ alg\`ebre dite la $preL_\infty$ alg\`ebre enveloppante de $(V,\diamond,d)$.\\
\end{proposition}


\subsection{Les alg\`ebres pr\'e-commutative (ou de Zinbiel) à homotopie près}

\

Une loi d'alg\`ebre pr\'e-commutative, ou d'alg\`ebre de Zinbiel, est une loi binaire dont le sym\'etris\'e est
 une loi commutative et associative. Plus pr\'ecis\'ement :\\

\begin{definition} {\rm \cite{[L1],[Liv]}}

\

On dit que $(V,\curlywedge,d)$ est une algèbre de Zinbiel (ou
pr\'e-commutative) \`a droite, différentielle et graduée si $V$
est un espace gradu\'e muni d'un produit $\curlywedge$ de degré
$0$ et d'une différentielle $d$ de degré $1$ vérifiant :
\begin{itemize}
\item $\forall x,y,z\in V$,~~$(x\curlywedge y)\curlywedge z=x\curlywedge (y\curlywedge z)+(-1)^{|y||z|}x\curlywedge (z\curlywedge y)$,
\item $\forall x,y\in V$,~~$d(x\curlywedge y)=dx\curlywedge  y+(-1)^{|x|}x\curlywedge dy$.\\
\end{itemize}
\end{definition}

Cette structure est associ\'ee \`a une op\'erade quadratique,
l'op\'erade $Zinb$. L'op\'erade duale, d\'etermin\'ee par
\cite{[L1],[Liv]}, est l'op\'erade Leibniz : $Zinb^!=Leib$.

Rappelons qu'une alg\`ebre de Leibniz (\`a droite) $(V,[~,~])$ est
un espace gradu\'e $V$ muni d'un crochet $[~,~]$ de degré $0$
vérifiant :
$$
\forall x,y,z\in V,~~[[x, y],z]=[x,[y, z]]+(-1)^{|y||z|}[[x,
z],y].
$$

\begin{definition} {\rm \cite{[Liv]}}

\

Une $Zinb^!$-cog\`ebre ou cogèbre de Leibniz est un espace
vectoriel gradué $\mathcal{C}$ muni d'une comultiplication
$\delta:\mathcal{C}\longrightarrow\mathcal{C}\otimes\mathcal{C}$
de degré $0$ vérifiant :
$$
(id\otimes \delta)\circ \delta=\Big(\delta\otimes
id-\tau_{23}\circ(\delta\otimes id)\Big)\circ \delta.
$$
\end{definition}

La cog\`ebre de Leibniz colibre engendr\'ee par $V[1]$ est donn\'ee par :\\

\begin{proposition} \cite{[Liv]}

\

Soit $V$ un espace vectoriel gradué. Alors la cogèbre de Leibniz
colibre graduée engendr\'ee par $V[1]$ est $(T^+(V[1]),\delta)$ où
$\delta$ est d\'efini par :
$$
\delta(x_1\otimes \dots \otimes x_n)=\sum_{1\leq k\leq
n-1}(x_1\otimes \dots \otimes
x_{k})\bigotimes\mu_{n-k}(x_{k+1}\otimes \dots \otimes x_{n}),
$$
les $\mu_j$ sont d\'efinis par r\'ecurrence ainsi : $\mu_1=id$,
et, si $\tau_n$ est le cycle $(1,\dots,n)$ de $S_n$,
$$
\mu_{n+1}=\mu_n\otimes id-(\mu_n\otimes id)\circ\tau_{n+1}^{-1}.
$$
(Comme pour la volte, l'action des $\mu_j$ sur les produits tensoriels est sign\'ee).\\
\end{proposition}

On peut montrer par r\'ecurrence :

\begin{lem}

\

Pour tout $p$, $q$ positif,
$$
\mu_{p+q}\circ sh_{p,q}=0.
$$
\end{lem}

\begin{definition}

\

Une structure de $Z_\infty$ alg\`ebre est la donn\'ee d'une
cog\`ebre de Leibniz codifférentielle $\Big(T^+(V[1]),\delta,
Q\Big)$ telle que $Q$
est une cod\'erivation de $\delta$ de degr\'e $1$ et de carr\'e nul.\\
\end{definition}

\begin{proposition} \cite{[Liv]}

\

Soit $(V,\curlywedge,d)$ une algèbre de Zinbiel diff\'erentielle
gradu\'ee. On pose
$$
Q_1(x)=dx,~Q_2(x\otimes y)=(-1)^{x}x\curlywedge y,~ Q_k=0,~\forall
k\geq3,
$$
et
$$\aligned
Q(x_1\otimes \dots \otimes x_n)&=\sum_{k=1}^{n}(-1)^{
\sum_{i<k}x_i}x_1\otimes...\otimes
x_{k-1}\otimes Q_1(x_k)\otimes x_{k+1}\otimes\dots\otimes x_n +\\
&+Q_2(x_1\otimes x_2)\otimes x_3\otimes\dots \otimes x_n+\\
&+\sum_{k=2}^{n-1}(-1)^{\sum_{i<k}x_i}x_1\otimes\dots\otimes
Q_2\circ \mu_2(x_k\otimes x_{k+1})\otimes\dots\otimes x_n.
\endaligned
$$
Alors, $\Big(T^+(V[1]),\delta, Q\Big)$ est une $Z_\infty$
alg\`ebre appel\'ee la $Z_\infty$ alg\`ebre enveloppante de $(V,\curlywedge,d)$.\\
\end{proposition}


\section{Les pré-$(a,b)$-alg\`ebres}

\begin{definition}

\

Soient $\mathcal{A}$ un espace vectoriel gradu\'e et $a,b\in
\mathbb{Z}$. On munit l'espace $\mathcal{A}$ d'un produit
$\curlywedge$ de degr\'e $a$ ($|\curlywedge|=a$) et d'un produit
$\lozenge$ de degr\'e $b$ ($|\lozenge|=b$) tel que
$\Big(\mathcal{A}[-a],\curlywedge\Big)$ soit une alg\`ebre
pré-commutative gradu\'ee  et $\Big(\mathcal{A}[-b],\lozenge\Big)$
soit une alg\`ebre pré-Lie gradu\'ee. C'est à dire:
$$(\alpha\curlywedge \beta)\curlywedge \gamma=\alpha\curlywedge (\beta\curlywedge
\gamma)+ (-1)^{(|\beta|+a)(|\gamma|+a)}\alpha\curlywedge
(\gamma\curlywedge \beta),$$
$$
(\alpha\lozenge \beta)\lozenge \gamma-\alpha\lozenge
(\beta\lozenge
\gamma)=(-1)^{(|\beta|+b)(|\gamma|+b)}\big((\alpha\lozenge
\gamma)\lozenge \beta -\alpha\lozenge(\gamma\lozenge\beta)\big).
$$

De plus, les produits $\curlywedge$ et $\lozenge$ vérifient les
relations de compatibilités suivantes:
$$\aligned
&\alpha\curlywedge(\beta\lozenge \gamma)=(-1)^{(|\beta|+b)(|\gamma|+b)}\alpha\curlywedge(\gamma\lozenge \beta)\\
&\alpha\lozenge(\beta\curlywedge \gamma)=(\alpha\lozenge \beta)\curlywedge\gamma\\
&(\alpha\lozenge\beta)\curlywedge
\gamma=(-1)^{(|\beta|+b)(|\gamma|+a)}(\alpha\curlywedge
\gamma)\lozenge\beta.
\endaligned
$$
 On dira que $\big(\mathcal{A},\curlywedge,\lozenge\big)$ est une
 pré-$(a,b)$-alg\`ebre (à droite) gradu\'ee.
 \end{definition}

Si on pose
$[\alpha,\beta]=\alpha\lozenge\beta-(-1)^{(|\alpha|+b)(|\beta|+b)}\beta\lozenge\alpha$
et
$\alpha\pt\beta=\alpha\curlywedge\beta+(-1)^{(|\alpha|+a)(|\beta|+a)}\beta\curlywedge\alpha$,
on obtient que $(\mathcal A,\pt,[~,~])$ est une $(a,b)$-alg\`ebre.
On aura aussi les deux relations suivantes :
$$\aligned
\alpha\curlywedge[\beta,\gamma]&=0\\
[\alpha,\beta\curlywedge\gamma]&=[\alpha,\beta]\curlywedge\gamma.
\endaligned
$$

\begin{ex}(\cite{[AAC2]})

\

 Une alg\`ebre pré-Gerstenhaber \`a droite graduée est un triplet $(\mathcal{G},\curlywedge,\lozenge)$ tel que :
\begin{itemize}
\item $(\mathcal{G},\curlywedge)$ est une algèbre de Zinbiel \`a droite graduée, $|\curlywedge|=0$.

\item $(\mathcal{G}[1],\lozenge)$ est une algèbre pré-Lie \`a droite gradu\'ee, $|\lozenge|=-1$.
\item  On impose les relations de compatibilité suivantes entre $\curlywedge$ et $\lozenge$ :
$$\aligned
&\alpha\curlywedge(\beta\lozenge \gamma)=(-1)^{(|\beta|-1)(|\gamma|-1)}\alpha\curlywedge(\gamma\lozenge \beta)\\
&\alpha\lozenge(\beta\curlywedge \gamma)=(\alpha\lozenge \beta)\curlywedge\gamma\\
&(\alpha\lozenge\beta)\curlywedge
\gamma=(-1)^{(|\beta|-1)|\gamma|}(\alpha\curlywedge
\gamma)\lozenge\beta.
\endaligned
$$
\end{itemize}

Dans ce cas, $(\mathcal G,\pt,[~,~])$ est une alg\`ebre de
Gerstenhaber.

\end{ex}

\begin{ex}

\

 Une alg\`ebre pré-Poisson \`a droite graduée est un triplet $(\mathcal{P},\curlywedge,\lozenge)$ tel que :
\begin{itemize}
\item $(\mathcal{P},\curlywedge)$ est une algèbre de Zinbiel \`a droite graduée, $|\curlywedge|=0$.

\item $(\mathcal{P},\lozenge)$ est une algèbre pré-Lie \`a droite gradu\'ee, $|\lozenge|=0$.
\item  On impose les relations de compatibilité suivantes entre $\curlywedge$ et $\lozenge$ :
$$\aligned
&\alpha\curlywedge(\beta\lozenge \gamma)=(-1)^{|\beta||\gamma|}\alpha\curlywedge(\gamma\lozenge \beta)\\
&\alpha\lozenge(\beta\curlywedge \gamma)=(\alpha\lozenge \beta)\curlywedge\gamma\\
&(\alpha\lozenge\beta)\curlywedge
\gamma=(-1)^{|\beta||\gamma|}(\alpha\curlywedge
\gamma)\lozenge\beta.
\endaligned
$$
\end{itemize}

Dans ce cas, $(\mathcal P,\pt,[~,~])$ est une alg\`ebre de Poisson
graduée. \end{ex}

\begin{ex}

\

Soit $\mathcal{A}$ l'espace des  formes différentielles sur une
variété $M$. On peut munir $\mathcal{A}$ de la graduation:
$|\alpha|=2k+3$, si $\alpha$ est un $k$-forme. L'espace
$\mathcal{A}$ est stable par le produit ext\'erieur $\wedge$. Si
$\alpha$ et $\beta$ sont deux formes différentielles de
$\mathcal{A}$, on définit:
$$
\alpha\curlywedge\beta=\displaystyle\frac{1}{|\beta|}\alpha\wedge
d\beta~~\qquad~\hbox{ et
}~\qquad~\alpha\lozenge\beta=\alpha\wedge\beta.
$$

Alors, on vérifie que $|\curlywedge|=-1$ et
$|\lozenge|=-3$.\vskip0.12cm

Pour $\alpha$, $\beta$ et $\gamma$ dans $\mathcal{A}$, on vérifie
aussi que :
$$
(\alpha\curlywedge \beta)\curlywedge
\gamma=\alpha\curlywedge(\beta\curlywedge\gamma)+(-1)^{(|\beta|-1)(|\gamma|-1)}\alpha\curlywedge(\gamma\curlywedge\beta),
$$
et
$$
(\alpha\lozenge \beta)\lozenge \gamma-\alpha\lozenge
(\beta\lozenge\gamma)=(-1)^{(|\beta|-3)(|\gamma|-3)}\Big((\alpha\lozenge\gamma)\lozenge
\beta -\alpha\lozenge(\gamma\lozenge\beta)\Big).
$$

Les relations de compatibilités entre $\curlywedge$ et $\lozenge$
sont aussi vérifiées :
$$\aligned
\alpha\curlywedge(\beta\lozenge \gamma)&=(-1)^{(|\beta|-3)(|\gamma|-3)}\alpha\curlywedge(\gamma\lozenge \beta)\\
\alpha\lozenge(\beta\curlywedge \gamma)&=(\alpha\lozenge \beta)\curlywedge\gamma\\
(\alpha\lozenge\beta)\curlywedge
\gamma&=(-1)^{(|\beta|-3)(|\gamma|-1)}(\alpha\curlywedge
\gamma)\lozenge\beta.
\endaligned
$$

Ainsi, $(\mathcal{A},\curlywedge,\diamond)$ est bien une pré-$(-1,-3)$-algèbre.\\

\end{ex}


\section{ L'alg\`ebre pr\'e-Lie diff\'erentielle $(\mathcal
H[a-b-1],R_2',D)$}

\

Soit $(\mathcal A,\curlywedge,\lozenge)$ une
pré-$(a,b)$-alg\`ebre. Puisque $\curlywedge$ est une loi
pr\'e-commutative, on lui a associ\'e une cog\`ebre de Leibniz
codiff\'erentielle $(\mathcal H,\delta,D)$. Dans cette section, on
montre
que $\mathcal H[a-b-1]$ est aussi muni d'une structure d'alg\`ebre pr\'e-Lie diff\'erentielle.\\

Pour cela, comme ci-dessus, on utilise un d\'ecalage de degré.
 On consid\`ere l'espace $\mathcal{A}[-a+1]$ muni de la graduation
 $dg(\alpha)=|\alpha|+a-1$ que l'on note simplement par $\alpha$. Sur $\mathcal{A}[-a+1]$, le
 produit $\curlywedge$ n'est plus de Zinbiel et le produit $\lozenge$ n'est pas
 pré-Lie.
  On construit, donc, un nouveau produit $\wedge$ sur $\mathcal{A}[-a+1]=\mathcal{A}[-a][1]$
 de degr\'e $1$ d\'efini
 par:
 $$\alpha\wedge\beta=(-1)^{1.\alpha}\alpha\curlywedge\beta,$$
et un nouveau produit $\diamond$ sur
$\mathcal{A}[-a+1]=\mathcal{A}[-b][b-a+1]$ de degr\'e $b-a+1$
d\'efini par:
$$\alpha\diamond\beta=(-1)^{(b-a+1).\alpha}\alpha\lozenge\beta.$$
Les produits $\wedge$ et $\diamond$ vérifient:
$$(-1)^{\alpha}(\alpha\wedge \beta)\wedge \gamma=-\alpha\wedge (\beta\wedge
\gamma)+ (-1)^{\beta\gamma}\alpha\wedge (\gamma\wedge \beta)$$
$$
(\alpha\diamond \beta)\diamond
\gamma-(-1)^{(b-a+1)(\alpha+1)}\alpha\diamond (\beta\diamond
\gamma)=(-1)^{\beta\gamma+b-a+1}\big((\alpha\diamond
\gamma)\diamond \beta
-(-1)^{(b-a+1)(\alpha+1)}\alpha\diamond(\gamma\diamond
\beta)\big).
$$
$$\aligned
&\alpha\wedge(\beta\diamond \gamma)=(-1)^{\beta\gamma+b-a+1}\alpha\wedge(\gamma\diamond \beta)\\
&\alpha\diamond(\beta\wedge \gamma)=(-1)^{\alpha+b-a+1}(\alpha\diamond \beta)\wedge\gamma\\
&(\alpha\diamond\beta)\wedge
\gamma=(-1)^{\beta\gamma+b-a+1}(\alpha\wedge \gamma)\diamond\beta.
\endaligned
$$
Définissons maintenant les produits $m_2$ et $\ell$ en posant :
$$
m_2(\alpha,\beta)=
 (-1)^{1.\alpha}\alpha\pt\beta
 =\alpha\wedge\beta-(-1)^{\alpha\beta}\beta\wedge\alpha$$
 $$\hbox{et }~~\ell(\alpha,\beta)=
(-1)^{(b-a+1).\alpha}[\alpha,\beta]=\alpha\diamond\beta-(-1)^{\alpha\beta+b-a+1}\beta\diamond\alpha.
$$
On aura aussi les deux relations suivantes :
$$\aligned
\alpha\wedge\ell(\beta,\gamma)&=0\\
\ell(\alpha,\beta\wedge\gamma)&=(-1)^{\alpha+b-a+1}\ell(\alpha,\beta)\wedge\gamma.
\endaligned
$$
De plus, $m_2$ et $\ell$ vérifient:
$$m_2(\alpha,\beta)=-(-1)^{\alpha
\beta}m_2(\beta,\alpha)$$
$$\ell(\alpha,\beta)=-(-1)^{b-a+1}(-1)^{\alpha
\beta}\ell(\beta,\alpha)$$
$$(-1)^{\alpha
\gamma}\ell(\ell(\alpha,\beta),\gamma\big)+(-1)^{\beta
\alpha}\ell(\ell(\beta,\gamma),\alpha) +(-1)^{\gamma
\beta}\ell(\ell(\gamma,\alpha),\beta)=0$$
$$
\ell(\alpha,m_2(\beta,\gamma))=(-1)^{\alpha+b-a+1}m_2(\ell(\alpha,\beta),\gamma)+
(-1)^{(\alpha+b-a+1)(\beta+1)} m_2(\beta,\ell(\alpha,\gamma))
$$

 On consid\`ere comme pr\'ec\'edemment
l'espace
 $$\mathcal{H}=\displaystyle\bigoplus_{n\geq1}\big(\bigotimes^n\mathcal{A}[-a+1]\big)=
 T^+(\mathcal{A}[-a+1])$$ et pour
$X=\alpha_1\otimes\dots\otimes\alpha_n\in\mathcal{H}$, le degr\'e
$dg(X)=\alpha_1+\dots+\alpha_n$ not\'e simplement par
$x$.\\

On prolonge $\wedge$ \`a $\mathcal{H}$ en $D$ de telle fa\c con
que $(\mathcal{H},\delta,D)$ soit une cogèbre de Leibniz
codifférentielle. Ce prolongement est donné par:
$$\aligned
D(\alpha_1\otimes \dots \otimes \alpha_n)&=(\alpha_1\wedge \alpha_2)\otimes \alpha_3\otimes\dots \otimes \alpha_n+\\
&+\sum_{k=2}^{n-1}(-1)^{\sum_{i<k}\alpha_i}\alpha_1\otimes\dots\otimes
m_2(\alpha_k, \alpha_{k+1})\otimes\dots\otimes \alpha_n.
\endaligned
$$

On prolonge après $\diamond$ \`a $\mathcal{H}$ en $R_2$. Si
$X=\alpha_1\otimes...\otimes\alpha_p$ et
$Y=\alpha_{p+1}\otimes...\otimes\alpha_{p+q}$ sont deux
\'el\'ements de $\mathcal H$, ce prolongement est donné par:
$$\aligned
&R_2(X,Y)=\\&=(\alpha_{1}\diamond\alpha_{p+1})\otimes\sum_{\sigma\in
Sh_{p-1,q-1}}\varepsilon_\alpha(\sigma^{-1})
(-1)^{(\alpha_2+\dots+\alpha_p)\alpha_{p+1}}\alpha_{\sigma^{-1}(2)}\otimes \dots\widehat{_{p+1}}\dots\otimes\alpha_{\sigma^{-1}(p+q)}\\
&+\hskip-0.5cm\sum_{\begin{smallmatrix}2\leq k\leq p\\ \sigma\in
Sh_{p-k,q-1}\end{smallmatrix}}\hskip-0.5cm\varepsilon_\alpha(\sigma^{-1})
(-1)^{(\alpha_{k+1}+\dots+\alpha_p)\alpha_{p+1}}(-1)^{(b-a+1)\sum_{s<k}\alpha_{s}}
\alpha_{1}\otimes\alpha_{2}\otimes\dots\otimes\alpha_{k-1}\otimes\\&\hskip2cm\otimes\ell(\alpha_{k},\alpha_{p+1})\otimes\alpha_{\sigma^{-1}(k+1)}\otimes\dots\widehat{_{p+1}}\dots\otimes\alpha_{\sigma^{-1}(p+q)}.
\endaligned
$$

On consid\`ere, maintenant, l'espace $\mathcal{H}[a-b-1]$ muni de
la graduation $dg'(X)=dg(X)-a+b+1$ not\'e simplement par $x'$ pour
$X\in\mathcal{H}[a-b-1]$. On pose
$R'_2(X,Y)=(-1)^{(a-b-1)dg'(X)}R_2(X,Y)$. Alors, le produit $R'_2$
est de degr\'e $0$ dans $\mathcal{H}[a-b-1]$ et la
diff\'erentielle $D$ reste de degr\'e $1$. Et on a

\begin{thm}

\

Le triplet $(\mathcal{H}[a-b-1],R_2',D)$ est une alg\`ebre pré-Lie diff\'erentielle gradu\'ee.\\
\end{thm}

\begin{proof}

\

La démonstration de ce théorème reprend celle du théorème de
\cite{[A]} pour les $(a,b)$-algèbres: les choix de signes sont
ceux de \cite{[A]} et les prolongements sont ceux des algèbres
pré-Gerstenhaber \cite{[AAC2]}. Dans la preuve, plusieurs cas
apparaissent. Par exemple, prouvons la relation $ (\ast)$ suivante
:
$$R_2'\big(R_2'(X,Y),Z\big)-R_2'\big(X,R_2'(Y,Z)\big)-(-1)^{y'z'}\hskip-0.18cm\Big(R_2'\big(R_2'(X,Z),Y\big)-R_2'\big(X,
R_2'(Z,Y)\big)\hskip-0.15cm\Big)\hskip-0.1cm=0.
$$

 Soient
$$\aligned
X&=\alpha_1\otimes\dots\otimes\alpha_p,\\
Y&=\alpha_{p+1}\otimes\dots\otimes\alpha_{p+q}\\
Z&=\alpha_{p+q+1}\otimes\dots\otimes\alpha_{p+q+r}
\endaligned
$$
trois \'el\'ements de $\mathcal H[a-b-1]$.

Dans la relation $(\ast)$, il appara\^\i t 4 types de termes :
\begin{itemize}
\item[1.] Dans $(\ast)$, il appara\^it des termes avec deux $\diamond$ :
$$
(\alpha_1\diamond\alpha_{p+1})\diamond\alpha_{p+q+1},\quad\alpha_1\diamond(\alpha_{p+1}\diamond\alpha_{p+q+1}),\quad(\alpha_1\diamond\alpha_{p+q+1})\diamond\alpha_{p+1},\quad
\alpha_1\diamond(\alpha_{p+q+1}\diamond\alpha_{p+1}).
$$
Ces termes apparaissent sous la forme $\pm
terme\otimes\alpha_{\sigma^{-1}(2)}\otimes\dots\alpha_{\sigma^{-1}(p+q+r)}$
o\`u $\sigma$ est un shuffle de
$\{1,\dots,p+q+r\}\setminus\{1,p+1,p+q+1\}$ : $\sigma\in
Sh_{p-1,q-1,r-1}$. Plus pr\'ecis\'ement, on pose :
$$
\varepsilon_\sigma=\varepsilon_\alpha(\sigma)(-1)^{(x-\alpha_1)\alpha_{p+1}}(-1)^{(x+y-\alpha_1-\alpha_{p+1})\alpha_{p+q+1}},
$$
et
$$
A_\sigma=\alpha_{\sigma^{-1}(2)}\otimes\dots\widehat{_{p+1}}\dots\widehat{_{p+q+1}}\dots\otimes\alpha_{\sigma^{-1}(p+q+r)}.
$$
La contribution des termes correspondants est
$C\otimes(\varepsilon_\sigma A_\sigma)$ avec :
$$\aligned
&C=(-1)^{(a-b-1)y'}\Big((\alpha_1\diamond\alpha_{p+1})\diamond\alpha_{p+q+1}-(-1)^{(b-a+1)(\alpha_1+1)}
\alpha_1\diamond(\alpha_{p+1}\diamond\alpha_{p+q+1})\\
&-(-1)^{\alpha_{p+1}\alpha_{p+q+1}+b-a+1}\big((\alpha_1\diamond\alpha_{p+q+1})\diamond\alpha_{p+1}-(-1)^{(b-a+1)(\alpha_1+1)}
\alpha_1\diamond(\alpha_{p+q+1}\diamond\alpha_{p+1})\big)\Big).
\endaligned
$$

Ces termes se simplifient gr\^ace \`a la relation vérifiée par $\diamond$.\\
\item[2.] Dans $(\ast)$, il appara\^it des termes avec un double crochet $\ell$ ou un $\diamond$ dans un crochet :
$$
\ell(\ell(\alpha_k,\alpha_{p+1}),\alpha_{p+q+1}),\
\ell(\alpha_k,\alpha_{p+1}\diamond\alpha_{p+q+1}),\ \ell(\ell(
\alpha_k,\alpha_{p+q+1}),\alpha_{p+1}),\
\ell(\alpha_k,\alpha_{p+q+1}\diamond\alpha_{p+1}).
$$
Ces termes apparaissent pour $1<k\leq p$ sous la forme
$$\aligned
&\pm (\alpha_1\otimes\dots\otimes\alpha_{k-1})\otimes terme\otimes\alpha_{\sigma^{-1}(k+1)}\otimes\dots\widehat{_{p+1}}\dots\widehat{_{p+q+1}}\dots\otimes\alpha_{\sigma^{-1}(p+q+r)}\\
=&\pm A_k\otimes terme\otimes B_\sigma.
\endaligned
$$
o\`u $\sigma$ est dans $Sh_{p-k,q-1,r-1}$ agissant sur
$\{k+1,\dots,p+q+r\}\setminus\{p+1,p+q+1\}$. On obtient donc les
termes $(-1)^{(a-b-1)y'}\varepsilon_\alpha(\sigma) A_k\otimes
C\otimes B_\sigma$ avec:
$$\aligned
&C=
\ell(\ell(\alpha_k,\alpha_{p+1}),\alpha_{p+q+1})-(-1)^{(\alpha_k+1)(b-a+1)}\ell(\alpha_k,\alpha_{p+1}\diamond
\alpha_{p+q+1})\\&-(-1)^{\alpha_{p+1}\alpha_{p+q+1}+b-a+1}\Big(\ell(\ell(
\alpha_k,\alpha_{p+q+1}),\alpha_{p+1})-(-1)^{(\alpha_k+1)(b-a+1)}
\ell(\alpha_k,\alpha_{p+q+1}\diamond\alpha_{p+1})\Big)\\
&=\ell(\ell(\alpha_k,\alpha_{p+1}),\alpha_{p+q+1})-(-1)^{(\alpha_k+1)(b-a+1)}\ell(\alpha_k,\ell
(\alpha_{p+1},\alpha_{p+q+1}))\\&\hskip2.8cm-
(-1)^{\alpha_{p+1}\alpha_{p+q+1}+b-a+1}\ell(\ell(\alpha_k,\alpha_{p+q+1}),\alpha_{p+1})\\
&=0.
\endaligned
$$

\item[3.] Dans $(\ast)$, il appara\^it des termes de la forme
$$
\dots\otimes\ell(\alpha_k,\alpha_{p+1})\otimes\dots\otimes\ell(\alpha_\ell,\alpha_{p+q+1})
\otimes\cdots,\quad\dots\otimes\ell(\alpha_\ell,\alpha_{p+q+1})\otimes\dots\otimes\ell(\alpha_k,\alpha_{p+1})\otimes\cdots.
$$
Plus pr\'ecis\'ement :

- Dans $R_2'\big(R_2'(X,Y),Z\big)$, pour tout $k\in\{2,\dots,p\}$,
les termes qui apparaissent sont :

$(1.1):
\dots\otimes\ell(\alpha_k,\alpha_{p+1})\otimes\dots\otimes\ell(\alpha_\ell,\alpha_{p+q+1})\otimes\cdots$
, avec $k<\ell\leq p$,

$(1.2):
\dots\otimes\ell(\alpha_k,\alpha_{p+1})\otimes\dots\otimes\ell(\alpha_\ell,\alpha_{p+q+1})\otimes\cdots$
, avec $p+1<\ell\leq p+q$,

$(1.3): \dots\otimes\ell(\alpha_\ell,\alpha_{p+q+1})\otimes\dots\otimes\ell(\alpha_k,\alpha_{p+1})\otimes\cdots$ , avec $1<\ell< k$.\\

- Dans $R_2'\big(X,R_2'(Y,Z)\big)$, pour tout $k\in\{2,\dots,p\}$,
les termes qui apparaissent sont :

$(2.1): \dots\otimes\ell(\alpha_k,\alpha_{p+1})\otimes\dots\otimes\ell(\alpha_\ell,\alpha_{p+q+1})\otimes\cdots$ , avec $p+1<\ell\leq p+q$.\\

- Dans $R_2'\big(R_2'(X,Z),Y\big)$, pour tout $k\in\{2,\dots,p\}$,
les termes qui apparaissent sont :

$(3.1):
\dots\otimes\ell(\alpha_k,\alpha_{p+q+1})\otimes\dots\otimes\ell(\alpha_\ell,\alpha_{p+1})\otimes\cdots$
, avec $k<\ell\leq p$,

$(3.2):
\dots\otimes\ell(\alpha_k,\alpha_{p+q+1})\otimes\dots\otimes\ell(\alpha_\ell,\alpha_{p+1})\otimes\cdots$
, avec $p+q+1<\ell\leq p+q+r$,

$(3.3): \dots\otimes\ell(\alpha_\ell,\alpha_{p+1})\otimes\dots\otimes\ell(\alpha_k,\alpha_{p+q+1})\otimes\cdots$ , avec $1<\ell< k$.\\

- Dans $R_2'\big(X,R_2'(Z,Y)\big)$, pour tout $k\in\{2,\dots,p\}$,
les termes qui apparaissent sont :

$(4.1): \dots\otimes\ell(\alpha_k,\alpha_{p+q+1})\otimes\dots\otimes\ell(\alpha_\ell,\alpha_{p+1})\otimes\cdots$ , avec $p+q+1<\ell\leq p+q+r$.\\

Il est clair que $(1.2)-(2.1)=0$ et $(3.2)-(4.1)=0$. En utilisant
la commutativité des battements, on vérifie que
$(1.1)=(3.3)$ et $(1.3)=(3.1)$.\\

\item[4.] Enfin, dans $(\ast)$, il appara\^it des termes de la forme
$$
\alpha_1\diamond\alpha_{p+1}\otimes\dots\otimes\ell(\alpha_k,\alpha_{p+q+1})\otimes\cdots,\quad
\alpha_1\diamond\alpha_{p+q+1}\otimes\dots
\otimes\ell(\alpha_k,\alpha_{p+1})\otimes\dots.
$$
Plus pr\'ecis\'ement,\\

- Dans $R_2'\big(R_2'(X,Y),Z\big)$, les termes qui apparaissent
sont :

$(1.1)':
\alpha_1\diamond\alpha_{p+1}\otimes\dots\otimes\ell(\alpha_k,\alpha_{p+q+1})\otimes\cdots$
, avec $1<k\leq p$,

$(1.2)':
\alpha_1\diamond\alpha_{p+1}\otimes\dots\otimes\ell(\alpha_k,\alpha_{p+q+1})\otimes\cdots$
, avec $p+1<k\leq p+q$,

$(1.3)': \alpha_1\diamond\alpha_{p+q+1}\otimes\dots\otimes\ell(\alpha_k,\alpha_{p+1})\otimes\cdots$ , avec $1<k\leq p$.\\

\n- Dans $R_2'\big(X,R_2'(Y,Z)\big)$, les termes qui apparaissent
sont:

$(2.1)': \alpha_1\diamond\alpha_{p+1}\otimes\dots\otimes\ell(\alpha_k,\alpha_{p+q+1})\otimes\cdots$ , avec $p+1<k\leq p+q$.\\

- Dans $R_2'\big(R_2'(X,Z),Y\big)$, les termes qui apparaissent
sont :

$(3.1)':
\alpha_1\diamond\alpha_{p+q+1}\otimes\dots\otimes\ell(\alpha_k,\alpha_{p+1})\otimes\cdots$
, avec $1<k\leq p$,

$(3.2)':
\alpha_1\diamond\alpha_{p+q+1}\otimes\dots\otimes\ell(\alpha_k,\alpha_{p+1})\otimes\cdots$
, avec $p+q+1<k\leq p+q+r$,

$(3.3)': \alpha_1\diamond\alpha_{p+1}\otimes\dots\otimes\ell(\alpha_k,\alpha_{p+q+1})\otimes\cdots$ , avec $1< k\leq p$.\\

- Dans $R_2'\big(X,R_2'(Z,Y)\big)$, les termes qui apparaissent
sont :

$(4.1)': \alpha_1\diamond\alpha_{p+q+1}\otimes\dots\otimes\ell(\alpha_k,\alpha_{p+1})\otimes\cdots$ , avec $p+q+1<k\leq p+q+r$.\\

Il est clair que $(1.2)'-(2.1)'=0$ et $(3.2)'-(4.1)'=0$. En
utilisant la commutativité du produit Shuffle,
 on vérifie que $(1.1)'=(3.3)'$ et $(1.3)'=(3.1)'$.\\
\end{itemize}

On montre de m\^eme que la différentielle $D$ est une dérivation
de $R_2'$. Pour $X=\alpha_1\otimes\dots\otimes\alpha_p$ et $
Y=\alpha_{p+1}\otimes\dots\otimes\alpha_{p+q}$, on vérifie que
$$
D\circ
R_2'(X,Y)=R_2'\big(D(X),Y\big)+(-1)^{x'}R_2'\big(X,D(Y)\big)
$$

\end{proof}

Maintenant, puisque $(\mathcal H[a-b-1],R_2',D)$ est une alg\`ebre
pr\'e-Lie diff\'erentielle gradu\'ee, on construit alors sa
$preL_\infty$ algèbre enveloppante $(\mathcal{H}[a-b]\otimes
S(\mathcal{H}[a-b]),\Delta,Q)$. Explicitement, dans $\mathcal
H[a-b]$, le degré est $deg''(X)=deg(X)-a+b=x''$, on pose
$$
R''_2(X,Y)=(-1)^{x''}R_2'(X,Y)~~~\text{ et
}~~\ell_2''(X,Y)=R_2''(X,Y)+(-1)^{x''y''}R_2''(Y,X).~
$$

On prolonge ensuite \`a $\mathcal{H}[a-b]\otimes
S(\mathcal{H}[a-b])$, $D$ et $R''_2$ en $m$ et $R$ par :
$$\aligned
m(X_0\otimes X_1\dots X_n)&=D(X_0)\otimes X_1\dots X_n+\\
&\hskip
0.2cm+(-1)^{x_0''}\sum_{j=1}^{n}\varepsilon_{x''}\Big(\begin{smallmatrix}
x_1\dots x_n\\ x_j~x_1\dots\hat{_j}\dots
x_n\end{smallmatrix}\Big)X_0\otimes
D(X_j).X_1\dots\widehat{_j}\dots X_n\\
&=D(X_0)\otimes X_1\dots X_n+(-1)^{x_0''}X_0\otimes m''(X_1\dots
X_n)
\endaligned
$$
où $m''$ est la codérivation donnée comme ci-dessus dans le cas
des $(a,b)$-algèbres.

De m\^eme,
$$\aligned
R(X_0\otimes X_1\dots
X_n)&=\sum_{i=1}^n\varepsilon_{x''}\Big(\begin{smallmatrix}
x_1\dots x_n\\ x_i~x_1\dots\hat{_i}\dots
x_n\end{smallmatrix}\Big)R_2''(X_0, X_i) \otimes
X_{1}\dots \widehat{_i}\dots X_{n}+\\
&\hskip-0.4cm
+(-1)^{x_0''}\sum_{i<j}\varepsilon_{x''}\Big(\begin{smallmatrix}
x_1\dots x_n\\ x_i~x_j~x_1\dots\hat{_i}\dots\hat{_j}\dots
x_n\end{smallmatrix}\Big) X_0\otimes
\ell_2''(X_i, X_j).X_{1}\dots\widehat{_i}\dots\widehat{_j}\dots X_{n}\\
&=\sum_{i=1}^n\varepsilon_{x''}\Big(\begin{smallmatrix} x_1\dots
x_n\\ x_i~x_1\dots\hat{_i}\dots
x_n\end{smallmatrix}\Big)R_2''(X_0, X_i) \otimes X_{1}\dots
\widehat{_i}\dots
X_{n}\\&\hskip-0.4cm+(-1)^{x_0''}X_0\otimes\ell''(X_1\dots X_n),
\endaligned
$$
où $\ell''$ est la codérivation donnée comme ci-dessus dans le cas
des $(a,b)$-algèbres.

Si on pose $Q=(m+R)$, alors $Q$ est une cod\'erivation de $\Delta$
vérifiant $deg''(Q)=1$ et $Q^2=0$.

Ainsi on obtient que $\Big(\mathcal{H}[a-b]\otimes
S(\mathcal{H}[a-b]),\Delta,Q=m+R\Big)$ est une cogèbre
permutative codifférentielle, c'est à dire, une $preL_\infty$ alg\`ebre.\\

\noindent Il reste \`a construire sur cette $preL_\infty$
alg\`ebre le coproduit $\kappa$ qui fera de
$\mathcal H[a-b]\otimes S(\mathcal H[a-b])$ une cog\`ebre de Leibniz. C'est le but de la section suivante.\\


\section{ Le coproduit $\kappa$}

\

D'abord, on a vu que $(\mathcal{H},\delta)$ est une cog\`ebre de
Leibniz. On rappelle que :
$$
\delta(X)=\delta(x_1\otimes \dots \otimes x_n)=\sum_{1\leq k\leq
n-1}(x_1\otimes \dots \otimes
x_{k})\bigotimes\mu_{n-k}(x_{k+1}\otimes \dots \otimes x_{n}).
$$
On note simplement le coproduit $\delta$ :
$$
\delta(X)=\sum_{U\otimes V=X}U\otimes \mu V.
$$

On se place maintenant dans $\mathcal{H}[a-b]$, le coproduit
$\delta$ sym\'etris\'e sera not\'e $\kappa$.

\begin{definition}

\

Sur $\mathcal H[a-b]$, on d\'efinit le cocrochet suivant :
$$
\kappa(X_0)=\sum_{U_0\otimes
V_0=X_0}(-1)^{(a-b)u_0''}\Big(U_0\bigotimes\mu
V_0+(-1)^{u_0''v_0''+a-b+1}\mu V_0\bigotimes U_0\Big)
$$
Ce cocrochet $\kappa$ se prolonge en un coproduit, toujours not\'e
$\kappa$, de degré $a-b$, d\'efini sur $\mathcal{H}[a-b]\otimes
S(\mathcal{H}[a-b])$ par :
$$\aligned
&\kappa(X_0\otimes X_1\dots X_n)=\sum_{\begin{smallmatrix}
U_0\otimes V_0= X_0\\ I\cup
J=\{1,\dots,n\}\end{smallmatrix}}(-1)^{(a-b)u_0''}\times\Big(\varepsilon_{x''}\Big(\begin{smallmatrix}
u_0v_0x_1\dots x_n\\ u_0~x_I~v_0~x_J\end{smallmatrix}\Big)U_0\otimes X_I\bigotimes\mu V_0.X_J\\
&+(-1)^{a-b+1}\varepsilon_{x''}\Big(\begin{smallmatrix} u_0v_0x_1\dots x_n\\
v_0~x_J~u_0~x_I\end{smallmatrix}\Big)\mu V_0\otimes X_J\bigotimes
U_0. X_I\Big)+(-1)^{(a-b)x_0''}X_0\otimes \delta''(X_1... X_n),
\endaligned
$$
o\`u $\delta''$ est le cocrochet sur $S^+(\mathcal{H}[a-b])$,
d\'efini comme dans la section sur les $(a,b)$-algèbres par :
\begin{align*}
&\delta''(X_1\dots X_n)=\sum_{\begin{smallmatrix}1\leq s\leq n\\
I\cup
J=\{1,\dots,n\}\setminus\{s\}\end{smallmatrix}}(-1)^{\sum_{i<s}(a-b)x_i''}\sum_{\begin{smallmatrix}U_s\otimes
V_s=X_s\\ U_s,V_s\neq\emptyset\end{smallmatrix}
}(-1)^{(a-b)u_s''}\times\cr&\times\left(\varepsilon_{x''}\left(\begin{smallmatrix}
x_1\dots x_n\\ x_I~u_s~v_s~x_J\end{smallmatrix}\right)X_I.
U_s\bigotimes \mu V_s .
X_J+(-1)^{a-b+1}\varepsilon_{x''}\left(\begin{smallmatrix} x_1\dots x_n\\
x_I~v_s~u_s~x_J\end{smallmatrix}\right)X_I.\mu V_s\bigotimes U_s .
X_J\right).
\end{align*}
\end{definition}

Maintenant $\kappa$ est un coproduit de degr\'e $a-b$ qui, en un
certain sens, prolonge le coproduit de Leibniz $\delta$, de
degr\'e 0, d\'efini sur $\mathcal H$. En fait, on peut dire que
$(\mathcal H[a-b]\otimes S(\mathcal H[a-b],\kappa)$ est une
cog\`ebre de Leibniz, en tenant compte de ce d\'ecalage de
degr\'e. C'est \`a dire :

\begin{proposition}

\

Le coproduit $\kappa$ vérifie:
$$
(-1)^{a-b}(id\otimes \kappa)\circ\kappa=\big(\kappa\otimes id+
\tau''_{23}\circ(\kappa\otimes id)\big)\circ\kappa.
$$
\end{proposition}

\begin{proof}

\

D'une part, on a

$$\aligned
&(-1)^{a-b}(id\otimes\kappa)\circ\kappa(X_0\otimes X_1... X_n)=\\
&=(id\otimes\kappa)\Big(\sum_{\begin{smallmatrix}U_0\otimes
V_0=X_0\\ I\cup
J=\{1,\dots,n\}\end{smallmatrix}}(-1)^{(a-b)(u_0''+1)}\times\Big(\varepsilon_{x''}\left(\begin{smallmatrix}u_0v_0x_1\dots x_n\\ u_0~x_I~v_0~x_J\end{smallmatrix}\right)U_0\otimes X_I\bigotimes\mu V_0.X_J+\\
&\hskip 1cm+(-1)^{a-b+1}\varepsilon_{x''}\left(\begin{smallmatrix}
u_0v_0x_1\dots x_n\\ v_0~x_J~u_0~x_I\end{smallmatrix}\right)\mu
V_0\otimes X_J\bigotimes U_0.
X_I\Big)\\&\hskip2cm+(-1)^{(a-b+1)x_0''}X_0\otimes \delta''(X_1\dots X_n)\Big)\\
&=\sum_{\begin{smallmatrix}U_0\otimes V_0=X_0\\ I\cup J=\{1,\dots,n\}\end{smallmatrix}}(-1)^{(a-b)(u_0''+1)}\times\Big
(\varepsilon_{x''}\left(\begin{smallmatrix}u_0v_0x_1\dots x_n\\ u_0~x_I~v_0~x_J\end{smallmatrix}\right)
U_0\otimes X_I\bigotimes\delta''(\mu V_0.X_J)+\\
&\hskip 1cm+(-1)^{a-b+1}\varepsilon_{x''}\Big(\begin{smallmatrix}
u_0v_0x_1\dots x_n\\ v_0~x_J~u_0~x_I\end{smallmatrix}\Big)\mu
V_0\otimes X_J\bigotimes \delta''(U_0. X_I)\Big)\\&\hskip2cm
+(-1)^{(a-b+1)x_0''}X_0\otimes (id\otimes\delta'')\circ\delta''(X_1\dots X_n)\\
&=(1.1)+(1.2),
\endaligned
$$
o\`u $(1,2)$ est le dernier terme, en
$X_0\otimes\delta''\circ(id\otimes\delta'')$.

D'autre part, on a
$$\aligned
&(\kappa\otimes id)\circ\kappa(X_0\otimes X_1\dots X_n)=\\
&=(\kappa\otimes id)\Big(\sum_{\begin{smallmatrix}U_0\otimes
V_0=X_0\\ I\cup
J=\{1,\dots,n\}\end{smallmatrix}}(-1)^{(a-b)u_0''}\times\Big(\varepsilon_{x''}\Big(\begin{smallmatrix}
u_0v_0x_1\dots x_n\\ u_0~x_I~v_0~x_J\end{smallmatrix}\Big)U_0\otimes X_I\bigotimes\mu V_0.X_J\\
&+(-1)^{a-b+1}\varepsilon_{x''}\Big(\begin{smallmatrix}
u_0v_0x_1\dots x_n\\ v_0~x_J~u_0~x_I\end{smallmatrix}\Big)\mu
V_0\otimes X_J\bigotimes U_0.
X_I\Big)+(-1)^{(a-b)x_0''}X_0\otimes \delta''(X_1\dots X_n)\Big)\\
&=\sum_{\begin{smallmatrix}U_0\otimes V_0=X_0\\ I\cup
J=\{1,\dots,n\}\end{smallmatrix}}(-1)^{(a-b)u_0''}\times\Big(\varepsilon_{x''}\Big(\begin{smallmatrix}
u_0v_0x_1\dots x_n\\ u_0~x_I~v_0~x_J\end{smallmatrix}\Big)
\kappa(U_0\otimes X_I)\bigotimes\mu V_0.X_J+\\
&+(-1)^{a-b+1}\varepsilon_{x''}\Big(\begin{smallmatrix} u_0v_0x_1\dots x_n\\
v_0~x_J~u_0~x_I\end{smallmatrix}\Big)\kappa(\mu V_0\otimes
X_J)\bigotimes U_0. X_I\Big)+\\&+(-1)^{(a-b)x_0''}\kappa(X_0)
\otimes \delta''(X_1\dots X_n)+(-1)^{(a-b)x_0''}X_0\otimes (\delta''\otimes id)\circ\delta''(X_1\dots X_n)\\
&=(2.1)+(2.2),
\endaligned
$$
o\`u $(2,2)$ est le dernier terme, en $X_0\otimes (\delta''\otimes
id)\circ\delta''$.

Et on a aussi :
$$\aligned
&\tau_{23}''\circ(\kappa\otimes id)\circ\kappa(X_0\otimes X_1\dots X_n)=\\
&=\tau_{23}''\Big(\sum_{\begin{smallmatrix}U_0\otimes V_0=X_0\\
I\cup
J=\{1,\dots,n\}\end{smallmatrix}}(-1)^{(a-b)u_0''}\times\Big(\varepsilon_{x''}\left(\begin{smallmatrix}
u_0v_0x_1\dots x_n\\ u_0~x_I~v_0~x_J\end{smallmatrix}\right)\kappa(U_0\otimes X_I)\bigotimes\mu V_0.X_J+\\
&\hskip 1cm+(-1)^{a-b+1}\varepsilon_{x''}\left(\begin{smallmatrix}
u_0v_0x_1\dots x_n\\ v_0~x_J~u_0~x_I\end{smallmatrix}
\right)\kappa(\mu V_0\otimes X_J)\bigotimes U_0.
X_I\Big)+\\&+(-1)^{(a-b)x_0''}\kappa(X_0)\otimes\delta''(X_1\dots
X_n)\Big)
+(-1)^{(a-b)x_0''}X_0\otimes \tau_{23}''\circ(\delta''\otimes id)\circ\delta''(X_1\dots X_n)\\
&=(3.1)+(3.2),
\endaligned
$$
o\`u $(3.2)$ est le terme en $X_0\otimes
\tau_{23}''\circ(\delta''\otimes id)\circ\delta''$.

L'identité de coJacobi est vérifiée par $\delta''$ : elle se montre comme pour les $(a,b)$-alg\`ebres (\cite{[A]}),
donc $(1.2)=(2.2)+(3.2)$.\\

Dans tous les termes restants, l'\'el\'ement $X_0$ a \'et\'e
coup\'e au moins une fois. Ces termes correspondent donc au cas où
on coupe deux fois $X_0$ par $\kappa$ et au cas où on coupe une
fois $X_0$ et une fois un des $X_t$ ($t>0$) par $\kappa$. On
vérifie comme dans \cite{[AAC2]} que $(1.1)=(2.1)+(3.1)$.
\end{proof}

En fait, les coproduits de Leibniz $\kappa$ et permutatif $\Delta$
ont des propri\'et\'es de compatibilit\'es, ce qui fait de
$(\mathcal H[a-b]\otimes S(\mathcal H[a-b]),\Delta,\kappa)$ une
bicog\`ebre au sens de Loday (\cite{[L2]}).

\begin{proposition}

\

Les coproduits $\Delta$ et $\kappa$ vérifient les relations de
compatibilité suivantes :
$$\aligned
(1):& ~(id \otimes
\kappa)\circ\Delta=(-1)^{a-b+1}\tau_{23} \circ(id\otimes \kappa)\circ\Delta,\\
(2):& ~(id\otimes \Delta)\circ\kappa=(\kappa\otimes id)\circ\Delta+\tau_{23}\circ(\kappa\otimes id)\circ\Delta,\\
(3):& ~(\Delta\otimes id)\circ\kappa=(id \otimes
\kappa)\circ\Delta+\tau_{23}\circ(\kappa\otimes id)\circ\Delta.
\endaligned
$$
\vskip0.15cm

\end{proposition}

\begin{proof}

\

\noindent $(1)$ On rappelle que :
$$
\Delta (X_0\otimes X_1\dots X_n)=X_0\bigotimes X_1\dots
X_n+X_0\otimes\Delta'(X_1\dots X_n).
$$
où $\Delta'(X_1\dots X_n)$ est le coproduit défini sur la cogèbre
cocommutative $S^{+}(\mathcal H[a-b])$. On a donc :
\begin{align*}
(id \otimes \kappa)\circ\Delta (X_0\otimes X_1\dots
X_n)=\Big(X_0\bigotimes\delta''+X_0\otimes\big(id \otimes
\delta''\big)\circ\Delta'\Big)(X_1\dots X_n)\end{align*} D'autre
part,
$$\aligned
\tau_{23}''\circ(id \otimes\kappa)\circ\Delta& (X_0\otimes X_1\dots X_n)=\\
&=\Big(X_0\bigotimes \tau_{23}''\circ\delta''+X_0\otimes\big(id
\otimes\tau_{23}''\circ\delta''\big)\circ\Delta'\Big)(X_1\dots
X_n)
\endaligned
$$
Comme $\tau_{23}''\circ\delta''=(-1)^{a-b+1}\delta''$. Ainsi, ~$
\tau_{23}'' \circ(id\otimes\kappa)\circ\Delta=(-1)^{a-b+1}(id
\otimes\kappa)\circ\Delta. $

\

\noindent $(2)$ D'une part, on a :
$$\aligned
&(id\otimes \Delta)\circ\kappa(X_0\otimes X_1\dots X_n)=\\
&=(id\otimes \Delta)\Big(\sum_{\begin{smallmatrix}U_0\otimes V_0=X_0 \\ I\cup J=\{1,\dots,n\}\end{smallmatrix}}
 (-1)^{(a-b)u_0''}\Big(\varepsilon_{x''}\left(\begin{smallmatrix}u_0~v_0~x_1\dots x_n\\ u_0~x_I~v_0~x_J\end{smallmatrix}
 \right)
 U_0\otimes X_I\bigotimes \mu V_0.X_J+\\
&+(-1)^{a-b+1}\varepsilon_{x''}\Big(\begin{smallmatrix}u_0~v_0~x_1\dots x_n\\
v_0~x_J~u_0~x_I\end{smallmatrix}\Big)\mu V_0\otimes X_J\bigotimes
U_0 .
X_I\Big)+(-1)^{(a-b)x_0''}X_0\otimes\delta''( X_1\dots X_n)\Big)\\
&=\hskip-0.7cm\sum_{\begin{smallmatrix}U_0\otimes V_0=X_0 \\ I\cup J\cup K=\{1,\dots,n\};K\neq\emptyset
\end{smallmatrix}} \hskip -0.7cm(-1)^{(a-b)u_0''}\Big\{\varepsilon_{x''}\Big(\begin{smallmatrix}u_0~v_0~x_1\dots x_n\\
 u_0~x_I~v_0~x_{J\cup K}\end{smallmatrix}\Big) \Big(\varepsilon_{x''}\left(\begin{smallmatrix}v_0~x_{J\cup K}\\
 v_0~x_J~x_ K\end{smallmatrix}\right) U_0\otimes X_I\bigotimes\mu V_0.X_J\bigotimes X_K\\
&\hskip1cm+\varepsilon_{x''}\Big(\begin{smallmatrix}v_0~x_{J\cup
K}\\ x_K~v_0~x_J\end{smallmatrix}\Big)
U_0\otimes X_I\bigotimes X_K\bigotimes \mu V_0.X_J\Big)+\\
&\hskip 1cm+(-1)^{a-b+1}\varepsilon_{x''}\Big(\begin{smallmatrix}
u_0~v_0~x_1\dots x_n\\ v_0~x_J~u_0~x_{I\cup
K}\end{smallmatrix}\Big)\Big(\varepsilon_{x''}\left(\begin{smallmatrix}
u_0~x_{I\cup K}\\ u_0~x_I~x_ K\end{smallmatrix}\right)\mu V_0\otimes X_J\bigotimes U_0.X_I\bigotimes X_K+\\
&\hskip 1cm+\varepsilon_{x''}\Big(\begin{smallmatrix}
u_0~x_{I\cup K}\\ x_K~u_0~x_I\end{smallmatrix}\Big)\mu V_0\otimes X_J\bigotimes X_K\bigotimes U_0.X_I\Big)\Big\}+\\
&\hskip 1cm+(-1)^{(a-b)x_0''}X_0\otimes(id\otimes \Delta')\circ\delta''( X_1\dots X_n)\\
&=(1.1)+(1.2)+(1.3)+(1.4)+(1.5).
\endaligned
$$

D'autre part, on a :
$$
\aligned
&(\kappa\otimes id)\circ\Delta(X_0\otimes X_1\dots X_n)=(\kappa\otimes id)\Big(X_0\bigotimes X_1\dots X_n+X_0\otimes\Delta'(X_1\dots X_n)\Big)\\
&=\sum_{U_0\otimes V_0=X_0
}(-1)^{(a-b)u_0''}\times\\&\times\Big(\Big(U_0\bigotimes\mu
V_0\bigotimes X_1 \dots X_n
+(-1)^{a-b+1}\varepsilon_{x''}\Big(\begin{smallmatrix} u_0~v_0\\
v_0~u_0\end{smallmatrix}\Big)\mu V_0\bigotimes U_0\bigotimes X_1\dots X_n\Big)\\
&\hskip 3cm+\sum_{\begin{smallmatrix} I\cup J\cup K =\{1,\dots,n\}\\ I\cup K\neq\emptyset,
J\neq\emptyset\end{smallmatrix}}\varepsilon_{x''}\left(\begin{smallmatrix} u_0~v_0~x_{I\cup J}~x_K\\
u_0~x_I~v_0~x_J~x_K\end{smallmatrix}\right) U_0\otimes X_I\bigotimes\mu V_0.X_J\bigotimes X_K+\\
&\hskip 3cm+(-1)^{a-b+1}\varepsilon_{x''}\Big(\begin{smallmatrix}
u_0~v_0~x_{I\cup J}~x_K\\
v_0~x_J~u_0~x_I~x_K\end{smallmatrix}\Big)
\mu V_0\otimes X_J\bigotimes U_0.X_I\bigotimes X_K\Big)+\\
&\hskip 3cm+(-1)^{(a-b)x_0''}X_0\otimes(\delta''\otimes
id)\circ\Delta'(X_1\dots X_n)\\&=\sum_{U_0\otimes
V_0=X_0}(-1)^{(a-b)u_0''}\sum_{\begin{smallmatrix} I\cup J\cup K
=\{1,\dots,n\}\\ K\neq\emptyset\end{smallmatrix}}
\varepsilon_{x''}\left(\begin{smallmatrix} u_0~v_0~x_1 \dots x_n\\
u_0~x_I~v_0~x_J~x_K\end{smallmatrix}\right)U_0\otimes
X_I\bigotimes
\mu V_0.X_J\bigotimes X_K+\\
&\hskip 3cm+(-1)^{a-b+1}\varepsilon_{x''}\left(\begin{smallmatrix}
u_0~v_0~x_1 \dots~x_n\\
v_0~x_J~u_0~x_I~x_K\end{smallmatrix}\right)\mu V_0\otimes
X_J\bigotimes
U_0.X_I\bigotimes X_K+\\
&\hskip 3cm+(-1)^{(a-b)x_0''}X_0\otimes(\delta''\otimes
id)\circ\Delta'(X_1\dots X_n)\\& =(2.1)+(2.2)+(2.3). \endaligned$$

De m\^eme,
$$
\aligned
&\tau_{23}''\circ(\kappa\otimes id)\circ\Delta(X_0\otimes X_1\dots X_n)=\\
&=\hskip-0.5cm\sum_{U_0\otimes V_0=X_0}(-1)^{(a-b)u_0''}\times\\&\times\sum_{\begin{smallmatrix} I\cup J\cup K =\{1,\dots,n\}\\
K\neq\emptyset\end{smallmatrix}}\hskip-0.5cm\varepsilon_{x''}\left(\begin{smallmatrix} u_0~v_0~x_1 \dots x_n\\
u_0~x_I~v_0~x_J~x_K\end{smallmatrix}\right)\varepsilon_{x''}\left(\begin{smallmatrix} v_0~x_J~x_K\\
x_K~v_0~x_J\end{smallmatrix}\right)U_0\otimes X_I\bigotimes X_K\bigotimes\mu V_0.X_J+\\
&\hskip 1cm+(-1)^{a-b+1}\varepsilon_{x''}\Big(\begin{smallmatrix} u_0~v_0~x_1 \dots~x_n\\
v_0~x_J~u_0~x_I~x_K\end{smallmatrix}\Big)\varepsilon_{x''}\left(\begin{smallmatrix} u_0~x_I~x_K\\
x_K~u_0~x_I\end{smallmatrix}\right)\mu V_0\otimes X_J\bigotimes X_K\bigotimes U_0.X_I+\\
&\hskip 1cm+(-1)^{(a-b)x_0''}X_0\otimes\tau_{23}''\circ(\delta''\otimes id)\circ\Delta'(X_1\dots X_n)\\
&=(3.1)+(3.2)+(3.3).
\endaligned
$$

On vérifie que $(1.5)=(2.3)+(3.3)$, gr\^ace \`a l'identit\'e de coLeibniz entre $\Delta'$ et $\delta''$,
\'etablie comme dans \cite{[A]}. De m\^eme, $(1.1)=(2.1)$, $(1.3)=(3.1)$, $(1.2)=(2.2)$ et $(1.4)=(3.2)$.\\

\noindent $(3)$ Cette relation se démontre comme la relation
précédente $(2)$.
\end{proof}

On a ainsi muni l'espace $\mathcal H[a-b]\otimes S(\mathcal H[a-b])$ d'une structure de bicog\`ebre permutative
et de Leibniz. On note cette bicog\`ebre $\Big(\mathcal H[a-b]\otimes S(\mathcal H[a-b]),\Delta,\kappa\Big)$.\\


\section{Pr\'e-$(a,b)$-alg\`ebre \`a homotopie pr\`es}

\

On va montrer que les cod\'erivations $m$ et $R$ de $\Delta$,
obtenues \`a partir des lois de $\mathcal A$, sont aussi des
cod\'erivations du coproduit $\kappa$. Par cons\'equent $m+R$ est
une codiff\'erentielle \`a la fois pour $\Delta$ et pour $\kappa$.
On posera donc :

\begin{definition}

\

Une pré-$(a,b)$-algèbre à homotopie près ou pré-$(a,b)_\infty$
algèbre est, pour le m\^eme opérateur $Q$, une cogèbre permutative
codiff\'erentielle $(\mathcal{C},\Delta,Q)$ et une cogèbre de
Leibniz codiff\'erentielle $(\mathcal{C},\kappa,Q)$ telle que les
deux coproduits $\Delta$ et $\kappa$ satisfassent les relations de
compatibilités suivantes :

$$\aligned
(id \otimes\kappa)\circ\Delta&=(-1)^{a-b+1}\tau_{23}'' \circ(id\otimes \kappa)\circ\Delta,\\
(id\otimes \Delta)\circ\kappa&=(\kappa\otimes id)\circ\Delta+\tau_{23}''\circ(\kappa\otimes id)\circ\Delta,\\
(\Delta\otimes id)\circ\kappa&=(id \otimes
\kappa)\circ\Delta+\tau_{23}''\circ(\kappa\otimes id)\circ\Delta.
\endaligned
$$
\end{definition}

\begin{proposition}

\

Soit $\mathcal A$ une pré-$(a,b)$-alg\`ebre. Sur la bicog\`ebre
$(\mathcal H[a-b]\otimes S(\mathcal H[a-b]),\Delta,\kappa)$,
 l'op\'erateur de degré $1$ et de carr\'e nul $Q=m+R$ est une cod\'erivation du coproduit
 $\kappa$,
 c'est à dire : $(Q\otimes id+ id\otimes
Q)\circ\kappa=(-1)^{a-b}\kappa\circ Q$.\\
\end{proposition}

\begin{proof}

\

\noindent 1. Montrons d'abord que $m$ est une cod\'erivation de
$\kappa$ : $(m\otimes id+ id\otimes
m)\circ\kappa=(-1)^{a-b}\kappa\circ m$.

On rappelle que :
$$\aligned
\kappa(X_0\otimes X_1\dots X_n)&=\sum_{\begin{smallmatrix}U_0\otimes V_0=X_0\\
I\cup
J=\{1,\dots,n\}\end{smallmatrix}}(-1)^{(a-b)u_0''}\times\Big(\varepsilon_{x'}\left(\begin{smallmatrix}
u_0v_0x_1\dots x_n\\ u_0~x_I~v_0~x_J\end{smallmatrix}\right)U_0\otimes X_I\bigotimes\mu V_0.X_J+\\
&\hskip
-2.5cm+(-1)^{a-b+1}\varepsilon_{x''}\left(\begin{smallmatrix}
u_0v_0x_1\dots x_n\\ v_0~x_J~u_0~x_I\end{smallmatrix}\right)\mu
V_0\otimes X_J\bigotimes U_0.
X_I\Big)+(-1)^{(a-b)x_0''}X_0\otimes\delta''(X_1\dots X_n)
\endaligned
$$
et
$$
m(X_0\otimes X_1\dots X_n)=D(X_0)\otimes X_1\dots
X_n+(-1)^{x_0''}X_0\otimes m'(X_1\dots X_n),
$$
o\`u on a not\'e $m'$ la codiff\'erentielle $m$ de $\Delta'$ et
$\delta''$ dans $S^+(\mathcal H[a-b])$.

En développant, on trouve alors 10 termes dans $(m\otimes id+ id\otimes m)\circ\kappa$ et
dans $(-1)^{a-b}\kappa\circ m$.\\

- D'une part $(m\otimes id+ id\otimes m)\circ\kappa(X_0\otimes
X_1\dots X_n)$ s'\'ecrit
$$\aligned
&\hskip-0.5cm\sum_{\begin{smallmatrix}U_0\otimes V_0=X_0\\ I\cup
J=\{1,\dots,n\}\end{smallmatrix}}\hskip-0.5cm(-1)^{(a-b)u_0''}\varepsilon_{x''}\left(\begin{smallmatrix}
u_0v_0x_1\dots x_n\\
u_0~x_I~v_0~x_J\end{smallmatrix}\right)\times\\&\hskip2cm\times\Big(D(U_0)\otimes
X_I\bigotimes\mu V_0.X_J+
(-1)^{u_0''}U_0\otimes m'(X_I)\bigotimes \mu V_0.X_J\Big)+\\
&\hskip
1cm+(-1)^{(a-b)u_0''}(-1)^{a-b+1}\varepsilon_{x''}\left(\begin{smallmatrix}u_0v_0x_1\dots
x_n\\
v_0~x_J~u_0~x_I\end{smallmatrix}\right)\times\\&\hskip2cm\times\Big(
D(V_0)\otimes X_J\bigotimes U_0.X_I+(-1)^{v_0''}\mu V_0\otimes m'(X_J)\bigotimes U_0.X_I\Big)+\\
&\hskip
1cm+(-1)^{(a-b)u_0''+x_I''+u_0''}\varepsilon_{x''}\left(\begin{smallmatrix}u_0v_0x_1\dots
x_n\\ u_0~x_I~v_0~x_J\end{smallmatrix}\right)
U_0\otimes X_I\bigotimes D(\mu V_0).X_J+\\
&\hskip
1cm+(-1)^{(a-b)u_0''+x_I''+u_0''+v''_0}\varepsilon_{x''}\left(\begin{smallmatrix}u_0v_0x_1\dots
x_n\\ u_0~x_I~v_0~x_J\end{smallmatrix}\right)
U_0\otimes X_I\bigotimes\mu V_0.m'(X_J)+\\
&\hskip
1cm+(-1)^{(a-b)u_0''+x_J''+v_0''}(-1)^{a-b+1}\varepsilon_{x''}\left(\begin{smallmatrix}u_0v_0x_1\dots
x_n\\ v_0~x_J~u_0~x_I\end{smallmatrix}\right)
\mu V_0\otimes X_J\bigotimes D(U_0).X_I+\\
&\hskip
1cm+(-1)^{(a-b)u_0''+x_J''+v_0''+u''_0}(-1)^{a-b+1}\varepsilon_{x''}\left(\begin{smallmatrix}u_0v_0x_1\dots
x_n\\ v_0~x_J~u_0~x_I\end{smallmatrix}\right)
\mu V_0\otimes X_J\bigotimes U_0.m'(X_I)+\\
&\hskip 1cm+(-1)^{(a-b)x_0''}D(X_0)\otimes \delta''(X_1\dots X_n)+\\
&\hskip 1cm+(-1)^{(a-b)x_0''}X_0\bigotimes(m'\otimes id+ id\otimes m')\circ\delta''(X_1\dots X_n)\\
&=(1.1)+\dots+(1.10).
\endaligned
$$

- De m\^eme, puisque $\mu D(V_0)=D(\mu V_0)$, et si
$X_0=U_0\otimes V_0$,
$$
D(X_0)=D(U_0)\otimes V_0+(-1)^{u_0''}U_0\otimes D(V_0)=U'_0\otimes
V_0+(-1)^{u''_0}U_0\otimes V'_0,
$$
et comme
$m'(X_I.X_J)=m'(X_I).X_J+(-1)^{x''_I}X_I.m'(X_J)=X_I'.X_J\pm
X_I.X'_J$, alors on d\'eveloppe $(-1)^{a-b}\kappa\circ
m(X_0\otimes X_1\dots X_n)$ en :

$$\aligned
&(-1)^{a-b}\sum_{\begin{smallmatrix}U_0\otimes V_0=X_0\\ I\cup
J=\{1,\dots,n\}\end{smallmatrix}}\Big((-1)^{(a-b)(u_0')''}\varepsilon_{x''}\left(\begin{smallmatrix}
u_0'v_0x_1\dots x_n\\ u_0'~x_I~v_0~x_J\end{smallmatrix}\right)U_0'\otimes X_I\bigotimes\mu V_0.X_J+\\
&\hskip
1cm+(-1)^{(a-b)(u'_0)''}(-1)^{a-b+1}\varepsilon_{x''}\left(\begin{smallmatrix}u_0'v_0x_1\dots
x_n\\ v_0~x_J~u_0'~x_I\end{smallmatrix}\right)
\mu V_0\otimes X_J\bigotimes U_0'.X_I+\\
&\hskip 1cm+(-1)^{(a-b)u_0''+u_0''}\varepsilon_{x''}\left(\begin{smallmatrix}u_0v_0'x_1\dots x_n\\ u_0~x_I~v_0'~x_J\end{smallmatrix}\right)U_0\otimes X_I\bigotimes\mu V_0'.X_J+\\
&\hskip
1cm+(-1)^{(a-b)u_0''+u''_0}(-1)^{a-b+1}\varepsilon_{x''}\left(\begin{smallmatrix}u_0v_0'x_1\dots
x_n\\ v_0'~x_J~u_0~x_I\end{smallmatrix}\right)
\mu V_0'\otimes X_J\bigotimes U_0.X_I+\\
&\hskip
1cm+(-1)^{(a-b)u_0''+x''_0}\varepsilon_{x''}\Big(\begin{smallmatrix}u_0v_0x_1\dots
x_n\\ u_0~x_I'~v_0~x_J\end{smallmatrix}\Big)
U_0\otimes m'(X_I)\bigotimes\mu V_0.X_J+\\
&\hskip 1cm+(-1)^{(a-b)u''_0+x_0''}(-1)^{a-b+1}\varepsilon_{x''}\left(\begin{smallmatrix}u_0v_0x_1\dots x_n\\ v_0~x_J~u_0~x_I'\end{smallmatrix}\right)\mu V_0\otimes X_J\bigotimes U_0.m'(X_I)+\\
&\hskip
1cm+(-1)^{(a-b)u_0''+x_I''+x_0''}\varepsilon_{x''}\left(\begin{smallmatrix}
u_0v_0x_1\dots x_n\\ u_0~x_I~v_0~x_J'\end{smallmatrix}\right)U_0\otimes X_I\bigotimes \mu V_0.m'(X_J)+\\
&\hskip
1cm+(-1)^{(a-b)u_0''+x_I''+x_0''}(-1)^{a-b+1}\varepsilon_{x''}\left(\begin{smallmatrix}u_0v_0x_1\dots
x_n\\ v_0~x_J'~u_0~x_I\end{smallmatrix}\right)
\mu V_0\otimes m'(X_J)\bigotimes U_0.X_I\Big)\\
&\hskip 1cm+(-1)^{a-b}(-1)^{(a-b)(x_0''+1)}D(X_0)\otimes \delta''(X_1\dots X_n)\\
&\hskip 1cm+(-1)^{a-b}(-1)^{(a-b+1)x_0''}X_0\bigotimes\delta''\circ m'(X_1\dots X_n)\\
&=(2.1)+\dots+(2.10).
\endaligned
$$

On a imm\'ediatement $(1.10)=(2.10)$ car on sait que $m'$ est une
cod\'erivation de $\delta''$ (comme dans \cite{[A]}). On a aussi
imm\'ediatement $(1.9)=(2.9)$. Les autres termes se simplifient
deux \`a deux suivant :
$$\begin{array}{llll}
(1.1)=(2.1),~~&(1.2)=(2.5),~~&(1.3)=(2.4),~~&(1.4)=(2.8),\\
(1.5)=(2.3),&(1.6)=(2.7),&(1.7)=(2.2),&(1.8)=(2.6).
\end{array}
$$

\noindent 2. La relation $(R\otimes id+ id\otimes
R)\circ\kappa=(-1)^{a-b}\kappa\circ R$ se démontre comme dans
\cite{[AAC2]}.
\end{proof}

\begin{thm}

\

Soit $(\mathcal A,\curlywedge,\lozenge)$ une
pré-$(a,b)$-alg\`ebre, notons $\mathcal H=T^+(\mathcal A[-a+1])$,
$\Delta$, $\kappa$ les coproduits et $Q=m+R$, la cod\'erivation
d\'efinis ci-dessus sur $\mathcal H[a-b]\otimes S^+(\mathcal
H[a-b])$, alors
$$
\Big(\mathcal{H}[a-b]\otimes
S(\mathcal{H}[a-b]),\Delta,\kappa,Q\Big)
$$ est une $pre$-$(a,b)_\infty$ alg\`ebre, appel\'ee la $pre$-$(a,b)$-alg\`ebre à homotopie près enveloppante de
 $(\mathcal A,\curlywedge,\lozenge)$.\\
\end{thm}


\end{document}